\documentclass{article}

\usepackage[utf8]{inputenc}     
\usepackage[T1]{fontenc}         
\usepackage{lmodern}             

\usepackage{enumitem}
\usepackage{amsmath}
\usepackage{bbm,bm}
\usepackage{amssymb}
\usepackage{amsthm}
\usepackage{graphicx}
\graphicspath{{save_repo/}}
\usepackage{mathtools}
\usepackage[margin=3cm]{geometry}
\usepackage{array}
\usepackage{braket}
\usepackage[numbers]{natbib}
\usepackage{algorithm}
\usepackage{algorithmic}

\usepackage{authblk}

\numberwithin{equation}{section} 

\usepackage{hyperref}
\usepackage{xcolor}
\hypersetup{
    colorlinks,
    linkcolor={red!50!black},
    citecolor={blue!50!black},
    urlcolor={blue!80!black}
}
\usepackage[capitalize,noabbrev]{cleveref}  

\theoremstyle{plain}
\newtheorem{theorem}{Theorem}[section]
\newtheorem{definition}[theorem]{Definition} 

\newtheorem{proposition}[theorem]{Proposition} 
\newtheorem{lemma}[theorem]{Lemma} 
\newtheorem{corollary}[theorem]{Corollary} 
\newtheorem{remark}[theorem]{Remark} 

\newcommand{\R}{\mathbb{R}}
\newcommand{\N}{\mathbb{N}}

\newcommand{\XX}{\mathcal{X}}
\newcommand{\YY}{\mathcal{Y}}

\newcommand{\PP}{\mathcal{P}}

\newcommand{\MM}{\mathcal{M}}
\newcommand{\CC}{\mathcal{C}}

\newcommand{\LL}{\mathcal{L}}

\newcommand{\dd}{\mathrm{d}}

\newcommand{\Pers}{\mathrm{Pers}}  
\newcommand{\diam}{\mathrm{diam}}
\newcommand{\upperdiag}{\Omega} 
\newcommand{\groundspace}{\Omega} 
\newcommand{\thediag}{{\partial \upperdiag}} 

\newcommand{\defeq}{\vcentcolon=}
\newcommand{\eqdef}{=\vcentcolon}

\newcommand{\Adm}{\mathrm{Adm}}

\newcommand{\cvweak}{\xrightarrow{w}}

\newcommand{\id}{\mathrm{id}}
\renewcommand{\epsilon}{\varepsilon}

\newcommand{\aprox}{\mathrm{aprox}}

\newcommand{\KL}{\mathrm{KL}}
\newcommand{\MMD}{\mathrm{MMD}}
\newcommand{\TV}{\mathrm{TV}}

\newcommand{\Sk}{\mathrm{Sk}}

\newcommand{\FG}{\mathrm{FG}}
\newcommand{\OT}{\mathrm{OT}}

\newcommand{\cdiag}{c_{\thediag}}

\DeclareMathOperator*{\argmin}{arg\,min}

\title{An Homogeneous Unbalanced Regularized Optimal Transport model with applications to Optimal Transport with Boundary}
\author{Théo Lacombe}
\affil{LIGM, Université Gustave Eiffel}
\date{}

\begin{document}
\maketitle

\begin{abstract}
	This work studies how the introduction of the entropic regularization term in unbalanced Optimal Transport (OT) models may alter their homogeneity with respect to the input measures. We observe that in common settings (including balanced OT and unbalanced OT with Kullback-Leibler divergence to the marginals), although the optimal transport cost itself is not homogeneous, optimal transport plans and the so-called Sinkhorn divergences are indeed homogeneous. 
	However, homogeneity does not hold in more general Unbalanced Regularized Optimal Transport (UROT) models, for instance those using the Total Variation as divergence to the marginals. 
	We propose to modify the entropic regularization term to retrieve an UROT model that is homogeneous while preserving most properties of the standard UROT model. 
	We showcase the importance of using our Homogeneous UROT (HUROT) model when it comes to regularize Optimal Transport with Boundary, a transportation model involving a spatially varying divergence to the marginals for which the standard (inhomogeneous) UROT model would yield inappropriate behavior.
\end{abstract}

\section{Introduction}

Optimal Transport (OT) literature can be traced back to the seminal work of Monge \citep{ot:monge1784seminal}, where Monge proposes a way to interpolate between two distributions of mass, represented by two probabilities measures $\alpha,\beta$ supported on some space $\groundspace$, while minimizing a cost representing the total effort spent to move each element of mass in $\alpha$ to a corresponding one in $\beta$. 
In its modern formulation due to Kantorovich \citep{ot:kantorovich1942translocation}, the OT problem is introduced as a linear program $\OT(\alpha,\beta) \defeq \min_\pi \iint c(x,y) \dd \pi(x,y)$ over transport plans $\pi \in \Pi(\alpha,\beta)$ that correspond to measures supported on $\groundspace \times \groundspace$ whose marginals are exactly $\alpha$ and $\beta$. 
Here, $c(x,y)$ denotes the cost of transporting some mass located at $x$ to $y$. 
When $\groundspace \subset \R^d$ is convex and $c(x,y) = \|x-y\|^p$, the infimum value reached (to the power $1/p$) defines a metric between probability measures supported on $\groundspace$ called the Wasserstein distance. 
In addition, any optimal $\pi \in \Pi(\alpha,\beta)$ induces an interpolation between $\alpha$ and $\beta$ by setting $\mu_t \defeq A_t \# \pi \defeq \pi(A_t^{-1}(X))$ where $A_t(x,y) \defeq (1-t)x + ty$ that turns out to be a geodesic between $\alpha$ and $\beta$ for the Wasserstein distance and can also be understood as the solution of the so-called continuity equation (see for instance \citep[Thm.~7.21]{ot:villani2008optimal} and \citep[\S 5.4]{otam}). 
More generally,  gradient flows induced by transportation problems are closely related to evolutionary equations \citep{ambrosio2005gradient}. 

Naturally, this physical interpretation suggests that optimal transport models should be homogeneous with respect to the input measures $\alpha$ and $\beta$: loosely speaking, encoding the mass of $\alpha$ and $\beta$ in grams or in kilograms should not change the structure of the solutions we obtain to describe the behavior of a physical system. 
Formally, it means that if $\pi$ is an optimal transport plan between $\alpha$ and $\beta$, we expect $\lambda \pi$ (or, at least, some scaled version of $\pi$) to be an optimal transport plan between $\lambda \alpha$ and $\lambda \beta$, for some scaling factor $\lambda > 0$. 
Fortunately, this clearly holds in the standard formulation of OT (the objective function and the constraints are linear). 
While this formulation is restricted to measures with the same total masses (and, by homogeneity, boils down to probability measures), models of Unbalanced OT (UOT) have been proposed to handle measures with possibly different total masses by relaxing the marginal constraints (see \citep{ot:chizat2015unbalanced,ot:liero2015optimalUOT} and \cref{subsec:background_UROT}). 
Of interest in this work and developed in \cref{sec:OT_with_boundary} is the framework of \emph{Optimal Transport with Boundary} (OTB) proposed by Figalli and Gigli \citep{ot:figalli2010newTransportationDistance} to model heat diffusion process with Dirichlet boundary conditions. 
Their model enables the comparison of measures with different total masses by allowing the transportation of any amount of mass to, and from, the boundary $\thediag$ of the domain $\groundspace$ provided we pay the corresponding cost $c(\cdot,\thediag)$. 
Here as well, all these models of UOT are homogeneous. 

\paragraph{} A parallel line of development---mainly popularized by the work of Cuturi \citep{ot:cuturi2013sinkhorn}---proposes to regularize the standard OT model between probability measures by adding an entropic regularization term $+\epsilon \KL(\pi | \alpha \otimes \beta)$ where $\epsilon > 0$ is a regularization parameter, and $\KL(\mu|\nu) = \int \log \left(\frac{\dd \mu}{\dd \nu}\right) \dd \mu$ denotes the Kullback-Leibler divergence (here, between probability measures). 
This approach was initially motivated by computational aspects: the resulting problem becomes strictly convex and can be solved efficiently using the Sinkhorn algorithm: a fixed-point algorithm that only involves matrix manipulations (hence usable efficiently on modern hardware as GPUs). 
Nonetheless, this model appears to be supported by strong theoretical properties, in particular through the introduction of an ``unbiased'' version called the Sinkhorn divergences \citep{ot:ramdas2017wasserstein,ot:feydy2019interpolating}, presented in \cref{subsec:background_UROT}. 
Unbalanced and Regularized OT have been mixed together in the works \citep{ot:chizat2018scaling,ot:sejourne2019sinkhorn} in a setting that covers most UOT models (though not directly the OTB one). 
However, the resulting Unbalanced Regularized OT model (UROT) may fail to be homogeneous, mostly because of the introduction of the (non-linear) term $\alpha \otimes \beta$. 
In particular, naive adaptations of \citep{ot:sejourne2019sinkhorn} to introduce an entropic regularization in the OTB model will suffer with heavy inhomogeneity, hindering its use in practice and calling for the development of an entropic regularization term that would preserve homogeneity.

\subsection*{Outline and Contributions}

This paper is organized in the following way:
\begin{itemize}
	\item \cref{sec:background} presents the background on OT theory on which this work relies, including its regularized and unbalanced variants. 
	\item \cref{sec:inhomogeneity_std} studies the (in)homogeneity properties of Unbalanced Regularized OT in its standard formulation. We prove in particular that in the natural settings of balanced OT and KL-penalized marginals, although the transport cost itself is not homogeneous, the corresponding Sinkhorn divergence appears to be homogeneous thanks to the addition of a ``mass bias'' proposed by Séjourné et al. It gives a new perspective in favor of the use of this ``unbiased'' formulation of entropic OT in these contexts. We show that, however, in a more general setting (for instance when using the Total Variation as the marginal penalty), homogeneity does not hold in the standard UROT model.
	\item \cref{sec:HUROT} introduces a model of \emph{Homogeneous} Unbalanced Regularized OT (HUROT). This models enjoys most of the properties of the standard one (UROT): it is solved by applying the Sinkhorn algorithm to renormalized measures, is continuous with respect to the weak convergence of measures, and the corresponding Sinkhorn divergence is positive without the need to introduce a mass bias term. 
	\item Eventually, \cref{sec:OT_with_boundary} introduces a model of Regularized OT with Boundary (ROTB). We showcase the importance of enforcing homogeneity in this model using the approach developed in \cref{sec:HUROT}. Of importance, we prove that the resulting ROTB model, in addition to the properties it shares with the HUROT model (continuity, positivity of the Sinkhorn divergence, etc.), implies the same notion of convergence as its unregularized counterpart, which legitimates our approach as a consistent way to regularize this spatially varying UOT model. 
\end{itemize}
Our code is publicly available at \url{https://github.com/tlacombe/homogeneousUROT}.

\section{Background}
\label{sec:background}

\subsection{Preliminary definitions and notation}

In this work, $\groundspace$ denotes a compact subset of $\R^d$, $c : \groundspace \times \groundspace \to \R_+$ is a cost function that is assumed to satisfy $c(x,x)=0$, to be symmetric, and Lipschitz continuous on $\groundspace$, typically $c(x,y) = \| x - y\|^2$. 
The set $\MM(\groundspace)$ denotes the space of (non-negative) Radon measures supported on $\groundspace$, and $\PP(\groundspace) = \{ \alpha \in \MM(\groundspace),\ m(\alpha) = 1\}$ denotes the subset of probability measures, that is measures of total mass $m(\alpha) \defeq \alpha(\groundspace) = 1$. 
With the exception of \cref{sec:OT_with_boundary}, we also assume that the total masses of the measures are finite. 

Given a measure $\alpha \in \MM(\groundspace)$ and a function $f \in \CC(\groundspace)$, we use the notation $\braket{\cdot,\cdot}$ to denote the duality product, that is 
\begin{equation}
	\braket{f,\alpha} \defeq \int_\groundspace f(x) \dd \alpha(x).
\end{equation}
Given a function $K : \groundspace \times \groundspace \to \R_+$ and $\mu,\nu \in \MM(\groundspace)$, we also introduce the notations 
\begin{align*}
	\braket{\mu,\nu}_K &\defeq \iint K(x,y) \dd \mu(x) \dd \nu(y), \\
	\| \mu - \nu\|^2_K &\defeq \braket{ \mu - \nu , \mu - \nu}_K.
\end{align*}
We say that $K$ defines a positive definite kernel when $\|\mu-\nu\|_K \geq 0$, with equality if and only if $\mu = \nu$. 
In the following, we assume that $K_\epsilon : (x,y) \mapsto e^{-\frac{c(x,y)}{\epsilon}}$ defines a positive definite kernel for any $\epsilon > 0$ (which holds if, for instance, $c(x,y) = \|x-y\|$ or $\|x-y\|^2$). 
We say that a sequence of measures $(\alpha_n)_n \in \MM(\groundspace)^\N$ converges \emph{weakly} toward some $\beta \in \MM(\groundspace)$, denoted by $\alpha_n \cvweak \beta$, if for any continuous (bounded) map $f$ one has $\braket{f,\alpha_n} \to \braket{f, \beta}$. 
Note that this implies $m(\alpha_n) \to m(\beta)$. 

A function $\varphi : [0,+\infty) \to [0,+\infty]$ is said to be an \emph{entropy function} if it is convex, lower-semi-continuous and satisfies $\varphi(1) = 0$. We also set the convention $\varphi(p) = +\infty$ whenever $p < 0$ and, in this work, we will only consider entropy functions that satisfy $\varphi(0) < \infty$. 
Of interest is its \emph{Legendre transform}, defined by $\varphi^* : q \mapsto \sup_{p \geq 0} pq - \varphi(p)$.

For any two measures $\alpha,\beta \in \MM(\groundspace)$ satisfying $\alpha \ll \beta$ (that is, $\forall X \subset \groundspace,\ \beta(X) = 0 \Rightarrow \alpha(X) = 0$), one can define the \emph{Radon-Nikodym derivative} $\frac{\dd \alpha}{\dd \beta} : \groundspace \to \R_+$ which is characterized by the relation $\alpha = \frac{\dd \alpha}{\dd \beta} \beta$. 
From this, an entropy function $\varphi$ can be used to define the $\varphi$-\emph{divergence}:
\begin{equation}
	D_\varphi(\alpha | \beta) \defeq \braket{\varphi \circ \frac{\dd \alpha}{\dd \beta} , \beta} = \int_\groundspace \varphi\left( \frac{\dd \alpha}{\dd \beta}(x) \right) \dd \beta(x).
\end{equation}

Among the notorious choices to define a $\varphi$-divergence, one has $\varphi(p) = p \log(p) - p + 1$, whose Legendre transform is $\varphi^*(q) = e^q - 1$, and which defines the so-called \emph{Kullback-Leibler} divergence $D_\varphi = \KL$. 
As another example that will play an important role in this work, the \emph{Total Variation} between measures can also be retrieved as a $\varphi$-divergence by taking $\varphi(p) = |1 - p|$, yielding $D_\varphi(\alpha | \beta) = \int_\groundspace |\dd \alpha(x) - \dd \beta(x)| \eqdef \TV(\alpha - \beta)$. 
Finally, the \emph{convex indicator function} is defined by $\imath_c(p) = 0$ if $p=1$, and $+\infty$ otherwise, so that $D_{\imath_c}(\alpha|\beta) = 0$ if $\alpha = \beta$, and $+\infty$ otherwise. 
Note that $\imath_c^* = \id$, the identity map.

Finally, a function $F : \XX \to \YY$ (for some Banach spaces $\XX,\YY$) is said to be $h$-homogeneous if there exists a constant $h > 0$ such that for any $(\lambda,x) \in \R \times \XX$ we have $F(\lambda x) = \lambda^h F(x)$. 
When $h=1$, we will simply say that $F$ is homogeneous. 

\subsection{Balanced regularized Optimal Transport}
\label{subsec:background_balancedOT}

Let $\alpha,\beta \in \PP(\groundspace)$ denote two probability measures. We denote by $\Pi(\alpha,\beta) \defeq \{ \pi \in \MM(\groundspace\times \groundspace),\ \pi(\cdot, \groundspace) = \alpha, \pi(\groundspace,\cdot) = \beta \}$ the corresponding \emph{set of transport plans between $\alpha$ and $\beta$}, that is the measures $\pi$ supported on $\groundspace\times \groundspace$ whose marginals $\pi_1,\pi_2$ are equal to $\alpha,\beta$, respectively. The optimal transport cost between $\alpha$ and $\beta$ is defined as
\begin{equation}\label{eq:vanillaOT}
	\OT(\alpha,\beta) \defeq \inf_{\pi \in \Pi(\alpha,\beta)} \braket{\pi,c},
\end{equation}
and any minimizer of this problem is said to be an optimal transport plan between the two measures. 

Though widely studied during the second-half of the 20\textsuperscript{th} century---we refer the interested reader to \citep{ot:villani2008optimal} for a thorough presentation---its use in real-life applications remained limited mostly due to its computational burden: in practical settings (where $\alpha,\beta$ are discrete probability measures supported on $\sim n$ points), \eqref{eq:vanillaOT} requires $\mathcal{O}(n^3 \log(n))$ operations to be solved. 

In 2013, Cuturi significantly contributed to popularize the practical use of OT (in particular in the machine learning community) by observing that its entropic regularized version can be solved efficiently on modern hardware \cite{ot:cuturi2013sinkhorn}, see \citep{ot:CuturiPeyre2017COT} for an extensive overview of the computational aspects of OT. 
In its modern form, this regularized problem reads, for a parameter $\epsilon > 0$,
\begin{align}
	\OT_\epsilon(\alpha,\beta) &\defeq \inf_{\pi \in \Pi(\alpha,\beta)} \braket{\pi,c} + \epsilon \KL(\pi | \alpha \otimes \beta), \label{eq:balanced_reg_OT_primal} \\ 
	&= \sup_{f,g \in \CC(\groundspace)} \braket{f,\alpha} + \braket{g,\beta} - \epsilon \braket{e^{\frac{f \oplus g - c}{\epsilon}} - 1, \alpha \otimes \beta} \label{eq:balanced_reg_OT_dual},
\end{align}
where \eqref{eq:balanced_reg_OT_primal} is referred to as the primal problem and \eqref{eq:balanced_reg_OT_dual} as its dual. 
It is worth noting that despite its appealing computational properties, $\OT_\epsilon$ does not define a proper divergence between probability measures. 
In particular, $\alpha \mapsto \OT_\epsilon(\alpha,\beta)$ is \textbf{not} minimized for $\alpha=\beta$. 
This phenomenon, called the \emph{entropic bias} \citep{ot:janati2020debiased}, can be corrected by introducing the associated \emph{Sinkhorn divergence} \citep{ot:ramdas2017wasserstein,ot:genevay2018learning}, defined by
\begin{equation}\label{eq:skdiv_balanced}
	\Sk_\epsilon(\alpha,\beta) \defeq \OT_\epsilon(\alpha,\beta) - \frac{1}{2}\OT_\epsilon(\alpha,\alpha) - \frac{1}{2} \OT_\epsilon(\beta,\beta).
\end{equation}
Deeply studied in \citep{ot:feydy2019interpolating}, it can be proved that $\Sk_\epsilon(\alpha,\beta) \geq 0$, with equality if, and only if, $\alpha = \beta$. 
In addition, while $\Sk_\epsilon(\alpha,\beta) \to \OT(\alpha,\beta)$ when $\epsilon \to 0$, one also has that $\Sk_\epsilon(\alpha,\beta) \to \iint c(x,y) \dd (\alpha - \beta)(x) \dd (\alpha - \beta)(y) \eqdef \MMD(\alpha,\beta)$ in the regime $\epsilon \to \infty$, the later quantity being referred to as the \emph{Maximum Mean Discrepancy} (MMD) between $\alpha$ and $\beta$ \citep{gretton2012kernel}, another type of divergence between probability measures (note that this does not hold if the total masses of the measures is not precisely equal to $1$). 
This observation sheds a new light on the role of the regularization parameter $\epsilon$ as a way to interpolate between two kind of distances between probability measures inducing a natural trade-off between computational efficiency (MMD) and geometric accuracy (OT).

\subsection{Unbalanced Sinkhorn Divergences}
\label{subsec:background_UROT}

The problems introduced in \cref{subsec:background_balancedOT} are restricted to probability measures or, slightly more generally, to measures $\alpha,\beta$ with the same total masses $m(\alpha) = m(\beta)$. 
Indeed, $\Pi(\alpha,\beta)$ is otherwise empty, making the problem infeasible. 
This setting is referred to as \emph{balanced} OT. 
One way to extend \eqref{eq:vanillaOT} to measures of different total masses is to relax the marginal constraints using a $\varphi$-divergence. 
The \emph{unbalanced} OT problem reads, for a given entropy function $\varphi$:
\begin{equation}\label{eq:unbalanced_non_reg_OT}
	\OT_\varphi(\alpha,\beta) = \inf_{\pi \in \MM(\groundspace \times \groundspace)} \braket{c,\pi} + D_\varphi(\pi_1 | \alpha) + D_\varphi(\pi_2 | \beta). 
\end{equation}

Following \citep{ot:chizat2018scaling}, unbalanced and regularized OT can be mixed together yielding the following problems, dual of each other:
\begin{align}
	\OT_{\epsilon,\varphi}(\alpha,\beta) &\defeq \inf_{\pi \in \MM(\groundspace \times \groundspace)} \braket{\pi,c} + D_\varphi(\pi_1|\alpha) + D_\varphi(\pi_2|\beta) + \epsilon \KL(\pi | \alpha \otimes \beta) \label{eq:UROT_primal} \\
	& = \sup_{f,g \in \CC(\groundspace)} \braket{-\varphi^*(-f),\alpha} + \braket{-\varphi^*(-g),\beta} - \epsilon \braket{e^{\frac{f \oplus g - c}{\epsilon}} - 1, \alpha \otimes \beta} \label{eq:UROT_dual}
\end{align}
In the following, we will refer to this formulation as the \emph{standard Unbalanced Regularized OT (UROT) model}. 
Note that setting $\varphi = \imath_c$ retrieves \eqref{eq:balanced_reg_OT_primal} (balanced regularized OT) and setting $\epsilon=0$ retrieves \eqref{eq:unbalanced_non_reg_OT} (unbalanced OT).

This model has been deeply studied in \citep{ot:sejourne2019sinkhorn}. 
In particular, authors prove that the dual problem \eqref{eq:UROT_dual} can be solved by iterating an adapted version of the Sinkhorn algorithm that reads \citep[Def.~3]{ot:sejourne2019sinkhorn} which consists of building a sequence $(f_t,g_t)_t$ defined by
\begin{equation} \label{eq:sinkhorn_algorithm_std}
	\begin{aligned}
		f_{t+1}(x) &= -\aprox_{\epsilon,\varphi^*}\left( \epsilon \log \braket{ e^{\frac{g_t - c(x,\cdot)}{\epsilon}} , \beta} \right), \\
		g_{t+1}(y) &= -\aprox_{\epsilon,\varphi^*} \left( \epsilon \log \braket{ e^{\frac{f_{t+1} - c(\cdot,y) }{\epsilon}} , \alpha} \right),
	\end{aligned}
\end{equation}
where $\aprox_{\epsilon,\varphi^*}$ is the \emph{anisotropic proximity operator} \citep[Def.~2]{ot:sejourne2019sinkhorn} associated to (the Legendre transform of) the divergence $\varphi$ defined by
\begin{equation}\label{eq:def_aprox}
	\aprox_{\epsilon,\varphi^*}(p) \defeq \argmin_{q \in \R} \epsilon e^{\frac{p-q}{\epsilon}} + \varphi^*(q).
\end{equation}
Crucially, a couple $(f,g) \in \CC(\groundspace)$ is optimal for \eqref{eq:UROT_dual} if, and only if, it is a fixed point of the map $(f_t,g_t) \mapsto (f_{t+1},g_{t+1})$ \citep[Prop.~8]{ot:sejourne2019sinkhorn}. 
Furthermore, any sequence $(f_t,g_t)_t$ built following \eqref{eq:sinkhorn_algorithm_std} is guaranteed to converge towards such a fixed point (that is, an optimal pair of potentials) under mild assumptions \citep[Thm.~1]{ot:sejourne2019sinkhorn} that are satisfied in this work (namely, $c$ must be Lipschitz continuous on $\groundspace$, and one must be able to restrict \eqref{eq:UROT_dual} to a compact subset of $\CC(\groundspace)$, which is possible in our settings of interest: $D_\varphi=D_{\imath_c}$, $\TV$ or $\KL$, see \citep[Lemmas 8, 9]{ot:sejourne2019sinkhorn}). 
Eventually, if $(f,g)$ is optimal for the dual problem \eqref{eq:UROT_dual}, then 
\begin{equation}\label{eq:relation_primal_dual_std}
	\pi \defeq \exp\left( \frac{f \oplus g - c}{\epsilon} \right) \alpha \otimes \beta
\end{equation}
is optimal for the primal problem \eqref{eq:UROT_primal}. 

Finally, the authors introduce the unbalanced Sinkhorn divergence between $\alpha$ and $\beta$:
\begin{equation}\label{eq:skdiv_unbalanced_sejourne}
	\Sk_{\epsilon,\varphi} (\alpha,\beta) \defeq \OT_{\epsilon,\varphi}(\alpha,\beta) - \frac{1}{2}\OT_{\epsilon,\varphi}(\alpha,\alpha) - \frac{1}{2} \OT_{\epsilon,\varphi}(\beta,\beta) + \frac{\epsilon}{2} (m(\alpha) - m(\beta))^2.
\end{equation}
They prove that this formulation enjoys most of the properties of its balanced counterpart \eqref{eq:skdiv_balanced}, in particular it is continuous with respect to the weak convergence, non-negative, satisfies $\Sk_{\epsilon,\varphi}(\alpha,\beta) = 0 \Leftrightarrow \alpha = \beta$, is convex with respect to each of its entries, and induces the same topology as weak convergence on the set $\MM_{\leq m}(\groundspace)$ of Radon measures with total mass uniformly bounded by $m > 0$, that is $\Sk_{\epsilon,\varphi}(\alpha_n,\alpha) \to 0 \Leftrightarrow \alpha_n \cvweak \alpha$. Note however that contrary to the balanced case, it does not converge to some sort of distance between $\alpha$ and $\beta$ when $\epsilon \to \infty$ \citep[Proposition 17]{ot:sejourne2019sinkhorn}.

\begin{remark}\label{remark:mass_bias_std}
The presence of the term $+\frac{\epsilon}{2} (m(\alpha) - m(\beta))^2$ in \eqref{eq:skdiv_unbalanced_sejourne}, called the \emph{mass bias}, is required to make the unbalanced Sinkhorn divergence non-negative (and convex). 
Intuitively, this term arises from the constant term $-\epsilon \braket{-1,\alpha \otimes \beta} = \epsilon m(\alpha) m(\beta)$ in \eqref{eq:UROT_dual}: while in the balanced case ($m(\alpha) = m(\beta)$), these terms cancel each other when computing the Sinkhorn divergence \eqref{eq:skdiv_balanced}, in the unbalanced case, they yield a constant term $\epsilon(m(\alpha)m(\beta) - \frac{1}{2} m(\alpha)^2 - \frac{1}{2} m(\beta)^2) = -\frac{\epsilon}{2}(m(\alpha) - m(\beta))^2$ that must be compensated by the mass bias term to ensure the good behavior of the model, in particular its non-negativity. 
\end{remark}

\section{Homogeneity and inhomogeneity in the standard model}
\label{sec:inhomogeneity_std}

In this section, we study the homogeneity properties of the standard  model \eqref{eq:UROT_primal} presented in \cref{subsec:background_UROT} with respect to the couple of input measures $(\alpha,\beta)$. 

First, let us stress that non-regularized OT, should it be balanced \eqref{eq:vanillaOT} or not \eqref{eq:unbalanced_non_reg_OT}, is homogeneous in $(\alpha,\beta)$, that is
\[ \OT_{\epsilon=0,\varphi}(\lambda \alpha,\lambda \beta) = \lambda \cdot \OT_{\epsilon=0,\varphi}(\alpha,\beta) \]
for any $\lambda \geq 0$. 
Furthermore, if $\pi$ is an optimal transport plan between $\alpha$ and $\beta$, then $\lambda \pi$ is an optimal transport plan between $\lambda \alpha$ and $\lambda \beta$. 
As mentioned in the introduction, this behavior is desirable as an optimal transport plan may be used as a way to interpolate between $\alpha$ and $\beta$, and it would be surprising that a change of scale in the masses of the measures induces a structural change in the interpolation between the two measures.

However, the addition of the entropic regularization term which, in the dual \eqref{eq:UROT_dual}, reads\\ $-\epsilon \braket{e^{\frac{f \oplus g - c}{\epsilon}} - 1,\alpha \otimes \beta}$ induces a seemingly peculiar behavior in terms of homogeneity. 
Namely, if we let
\begin{equation}
	J_{(\alpha,\beta)}(f,g) \defeq \braket{-\varphi^*(-f),\alpha} + \braket{-\varphi^*(-g),\beta} -\epsilon \braket{e^{\frac{f \oplus g - c}{\epsilon}} - 1,\alpha \otimes \beta},
\end{equation}
one has
\[ J_{(\lambda \alpha, \lambda \beta)}(f,g) = \lambda \braket{-\varphi^*(-f),\alpha} + \lambda \braket{-\varphi^*(-g),\beta} - \lambda^2 \epsilon \braket{e^{\frac{f \oplus g - c}{\epsilon}} - 1,\alpha \otimes \beta}, \]
inducing a quadratic term in $\lambda$ that may hinder homogeneity. 

The goal of this section is to investigate the impact of this apparent inhomogeneity in the standard UROT model \eqref{eq:UROT_primal}. 

\subsection{The balanced case}
\label{subsec:homogene_balanced}
We first consider the case of regularized balanced optimal transport \eqref{eq:balanced_reg_OT_primal}; where $\varphi = \imath_c$. 
The following lemma describes the effect of a scaling of the measures on the sequence of potentials produced by the Sinkhorn algorithm \eqref{eq:sinkhorn_algorithm_std}.

\begin{lemma}\label{lemma:sk_update_balanced}
Let $\alpha,\beta \in \MM(\groundspace)$ be two measures of total mass $m(\alpha) = m(\beta) = m$. 
Fix $(f_0,g_0) \in \CC(\groundspace)$ and let $(f_t,g_t)_{t\geq 1}$ denote the sequence of dual potentials produced iterating \eqref{eq:sinkhorn_algorithm_std} starting from $(f_0, g_0)$ for the couple $(\alpha,\beta)$. 
Let $(f_t^{(\lambda)}, g^{(\lambda)}_t)_t$ denote the sequence produced starting from $(f_0, g_0)$ for the couple $(\lambda \alpha,\lambda \beta)$. 
Then, for all $t \geq 1$,
\[ (f_t^{(\lambda)}, g^{(\lambda)}_t) = (f_t - \epsilon \log(\lambda), g_t). \]
\end{lemma}

Hence, scaling the measures by a factor $\lambda$ reflects as a shift of $-\epsilon \log(\lambda)$ in the first potential of the sequence produced by the Sinkhorn algorithm, yielding a series of results summarized in the following corollary.

\begin{corollary} \label{corollary:balanced_homogeneity}
Let $\alpha,\beta \in \MM(\groundspace)$ be two measures of total mass $m(\alpha) = m(\beta) = m$.
\begin{enumerate}
\item If $(f,g)$ is a couple of optimal potentials for the dual problem for the couple $(\alpha,\beta)$, then $(f-\epsilon \log(\lambda), g)$ is optimal for $(\lambda \alpha, \lambda \beta)$.
\item If $\pi$ is an optimal transport plan for the couple $(\alpha,\beta)$, then $\lambda \pi$ is optimal for the couple $(\lambda \alpha,\lambda \beta)$. 
\item We have
\begin{equation}
 \OT_\epsilon(\lambda \alpha,\lambda \beta) = \lambda \cdot \OT_\epsilon(\alpha,\beta) + \epsilon \lambda (\lambda - 1) m^2 - \epsilon \log(\lambda) \lambda m,
\end{equation}
that is, the optimal transport cost is not homogeneous. 
\item We have
\begin{equation}
\Sk_\epsilon(\lambda \alpha,\lambda \beta) = \lambda \cdot \Sk_\epsilon(\alpha,\beta),
\end{equation}
that is, the Sinkhorn divergence is homogeneous in the balanced case. 
\end{enumerate}
\end{corollary}

Overall, the quantities of interest behave in a reasonable way, in particular the solutions of the primal problem are homogeneous. 
Interestingly, the Sinkhorn divergence cancels the inhomogeneous behavior appearing in $\OT_\epsilon$, giving an additional argument in favor of using this debiased (and homogenized) quantity to compare probability measures using regularized OT.

\begin{remark}\label{remark:warning_numerical}
We warn the reader interested in computational OT that the inhomogeneity appearing in $\OT_\epsilon$ may lead to ill-behavior in numerical applications. 
Indeed, in practice, the Sinkhorn algorithm \eqref{eq:sinkhorn_algorithm_std} does not exactly reach a fixed point and is instead run until some stopping criterion is reached. 
For instance, one may stop the iterations when the relative change in the objective value $v_t \defeq J_{(\alpha,\beta)}(f_t,g_t)$ is smaller than some $\tau > 0$, that is when $\left| \frac{v_{t+1} - v_t}{v_t} \right| < \tau$. 
However, the inhomogeneous behavior in $v_t$ implies that for a given $\tau$, the number of iterations needed to reach the criterion when comparing $\alpha$ and $\beta$ may differ from the one needed when comparing $\lambda \alpha$ and $\lambda \beta$. 
Thus, even though in theory the (optimal) transportation plans of both couples should be the same (up to the scaling factor $\lambda$), the numerical outputs (transport plan, Sinkhorn divergence, etc.) provided by the Sinkhorn algorithm may not satisfy this property. 
\end{remark}

\begin{proof}[Proof of \cref{lemma:sk_update_balanced}]
In this context, $\varphi^* = \id$ and subsequently, $\aprox_{\epsilon,\varphi^*} = \id$, hence the Sinkhorn iterations \eqref{eq:sinkhorn_algorithm_std} simply read
\begin{align*}
	f_{t+1} &= -\epsilon \log \braket{e^{\frac{g_t - c}{\epsilon}},\beta}, \\
	g_{t+1} &= -\epsilon \log \braket{e^{\frac{f_{t+1} - c}{\epsilon}},\alpha}.
\end{align*}
We observe that $f_1^{(\lambda)} = -\epsilon \log \braket{e^{\frac{g_0 - c}{\epsilon}},\lambda \beta} = f_1 - \epsilon \log(\lambda)$. 
Therefore, $g_1^{(\lambda)} = - \epsilon \log \braket{e^{\frac{f_{1} - \epsilon \log(\lambda) - c}{\epsilon}}, \lambda \alpha} = g_{1}$, and thus $f_2^{(\lambda)} = -\epsilon \log \braket{e^{\frac{g_1 - c}{\epsilon}},\lambda \alpha} = -\epsilon \log \braket{e^{\frac{g_1 - c}{\epsilon}},\alpha} - \epsilon \log(\lambda) = f_2 - \epsilon \log(\lambda)$. 
A simple induction gives the conclusion.
\end{proof}

\begin{proof}[Proof of \cref{corollary:balanced_homogeneity}]~
\begin{enumerate}
	\item Since the sequence of potentials $(f_t,g_t)_t$ converges to $(f,g)$ which are optimal for $(\alpha,\beta)$, it follows from \cref{lemma:sk_update_balanced} that $(f_t^{(\lambda)},g_t^{(\lambda)}) \to (f - \epsilon \log(\lambda), g)$ which must also be a fixed point of the Sinkhorn loop, hence a pair of optimal potentials for the couple $(\lambda \alpha,\lambda \beta)$. 
	\item Using the primal-dual relation \eqref{eq:relation_primal_dual_std}, we know that $\pi = e^{\frac{f \oplus g - c}{\epsilon}} \dd \alpha \otimes \beta$ is optimal for the couple $(\alpha,\beta)$. Therefore, from the previous point,
	\[
		\exp\left( \frac{f \oplus g - \epsilon\log(\lambda) - c }{\epsilon} \right) \dd (\lambda \alpha \otimes \lambda \beta) = \lambda \exp\left( \frac{f \oplus g - c}{\epsilon} \right) \dd \alpha \otimes \beta = \lambda \pi
	\]
	is optimal for the couple $(\lambda \alpha,\lambda \beta)$.
	\item Using $(f - \epsilon \log(\lambda),g)$ in the dual relation \eqref{eq:UROT_dual}, we have
\begin{align*} 
\OT_\epsilon(\lambda \alpha,\lambda \beta) &= \lambda \braket{f,\alpha} - \epsilon \log(\lambda) \lambda m(\alpha) + \lambda \braket{g, \beta} - \lambda \epsilon \braket{ e^{ \frac{f \oplus g - c}{\epsilon}} - \lambda , \alpha \otimes \beta}, \\
&= \lambda \left( \braket{f,\alpha} + \braket{g, \beta} - \epsilon \braket{ e^{ \frac{f \oplus g - c}{\epsilon}} -1 + (1 - \lambda) , \alpha \otimes \beta} \right) - \epsilon \log(\lambda) \lambda m(\alpha) \\
&= \lambda \OT_\epsilon(\alpha,\beta) + \epsilon \lambda (\lambda - 1) m(\alpha) m(\beta) - \epsilon \log(\lambda) \lambda m(\alpha) \\
&= \lambda \OT_\epsilon(\alpha,\beta) + \epsilon \lambda (\lambda - 1) m^2 - \epsilon \log(\lambda) \lambda m.
\end{align*}	
	\item The homogeneity of $\Sk_\epsilon$ follows from the fact that the ``inhomogeneous terms'' $+ \epsilon \lambda (\lambda - 1) m^2 - \epsilon \log(\lambda) \lambda m$ cancel in the definition of the balanced Sinkhorn divergence. 
\end{enumerate}
\end{proof}

\subsection{The KL case}

We now propose to derive the same study as the one of \cref{subsec:homogene_balanced} using $\varphi(p) = p \log(p) - p + 1$, that is $D_\varphi = \KL$, a common choice in unbalanced optimal transport to penalize the violation of the marginal constraints. 
In this context, $\varphi^*(q) = e^q - 1$, and $\aprox_{\epsilon,\varphi^*}(p) = \frac{1}{1+\epsilon} p$. 
The following proposition summarizes the important properties of this model as far as homogeneity is concerned.

\begin{proposition}\label{prop:homogeneity_KL_case}
Let $\alpha,\beta \in \MM(\groundspace)$. Then,
\begin{enumerate}
\item If $(f,g)$ is a pair of optimal potentials for the couple $(\alpha,\beta)$, then \\ $\left(f - \frac{\epsilon^2}{(1+\epsilon)^2 - 1} \log(\lambda), g - \frac{\epsilon^2}{(1+\epsilon)^2 - 1} \log(\lambda)\right)$ is optimal for the couple $(\lambda \alpha,\lambda \beta)$. 
\item If $\pi$ is an optimal transport plan for $(\alpha,\beta)$, then $\lambda^{h} \pi$ is optimal for $(\lambda \alpha,\lambda \beta)$, where $h = 2 - \frac{2}{2 + \epsilon}$.
\item The Sinkhorn divergence $\Sk_{\epsilon,\varphi}$ is $h$-homogeneous (while $\OT_{\epsilon,\varphi}$ is not $h$-homogeneous).
\end{enumerate}
\end{proposition}

As in the balanced case, the conclusions here are mostly positive: though the optimization problem \eqref{eq:UROT_dual} itself is not ($h$-)homogeneous, the optimal transport plans are $h$-homogeneous, and so is the Sinkhorn divergence (thanks to the addition of the mass bias term!). 

\begin{proof}
	As in the balanced case, we first investigate the behavior of the Sinkhorn algorithm under rescaling of the measures. Let $(f_0, g_0) \in \CC(\groundspace)$, let $(f_t, g_t)_t$ denote the sequence obtained when iterating the Sinkhorn loop for the couple $(\alpha, \beta)$ initialized at $(f_0,g_0)$, and let $(f_t^{(\lambda)}, g_t^{(\lambda)})_t$ be the one obtained for the couple $(\lambda \alpha,\lambda \beta)$ with the same initialization. We prove the following by induction: 
	\begin{align*}
		f^{(\lambda)}_t &= f_t - \epsilon u_t \log(\lambda) \\
		g^{(\lambda)}_t &= g_t - \epsilon v_t \log(\lambda),
	\end{align*}
	where $(u_t, v_t) \in \R \times \R$ are real sequences following the relations $u_{t+1} = T(v_t)$ and $v_{t+1} = T(u_{t+1})$ with $T(x) = \frac{1 - x}{1 + \epsilon}$, with $u_0, v_0 = 0$. Indeed, 
	\begin{align*}
	f^{(\lambda)}_{t+1} &= -\frac{\epsilon}{1 + \epsilon } \log \int e^{\frac{g_t^{(\lambda)} - c}{\epsilon}} \dd \lambda \beta \\
						&= - \frac{\epsilon}{1 + \epsilon} \log \int e^{\frac{g_t - c}{\epsilon}} \lambda^{1 - v_t} \dd \beta \\
						&= f_t - \epsilon \frac{1 - v_t}{1 + \epsilon} \log(\lambda).
	\end{align*}
	A similar computation holds for the second potentials $(g_t)_t$. 
	
	The sequences $(u_t)_t$ and $(v_t)_t$ converge to the fixed point of $T \circ T$, given by
	\[ x = T \circ T(x) \Leftrightarrow x = \frac{x + \epsilon}{(1 + \epsilon)^2} \Leftrightarrow x = \frac{\epsilon}{(1 + \epsilon)^2 - 1} = \frac{1}{2 + \epsilon},\]
	proving the result linking $(f_t^{(\lambda)}, g_t^{(\lambda)})$ and $(f_t, g_t)$.

From this, simple computations prove the claims:
\begin{enumerate}	
\item Follows from the fact that $(f_t,g_t)_t$ converges to a couple of optimal dual potentials for $(\alpha,\beta)$.
\item Follows from the fact that 
\[ \exp\left( \frac{f \oplus g - 2 \frac{\epsilon}{2 + \epsilon}\log(\lambda) - c}{\epsilon} \right) \lambda^2 \dd \alpha \otimes \beta = \exp\left( \frac{f \oplus g - c}{\epsilon} \right) \lambda^{- \frac{2}{2 + \epsilon} } \lambda^2 \dd \alpha \otimes \beta = \lambda^h \pi \]
is an optimal transport plan for the couple $(\lambda \alpha,\lambda \beta)$ (see \eqref{eq:relation_primal_dual_std}). 
\item The shift in the potentials induces a change in the objective value $J_{(\lambda\alpha,\lambda\beta)}$ reading
	\begin{align*}
		&\braket{1 - e^{-f + \frac{\epsilon}{2 + \epsilon} \log(\lambda)} , \lambda \alpha} + \braket{1 - e^{-g + \frac{\epsilon}{2 + \epsilon} \log(\lambda)} , \lambda \beta} - \epsilon \braket{ e^{ \frac{f \oplus g - c - 2 \frac{\epsilon}{2 + \epsilon} \log(\lambda)}{\epsilon} } - 1 , \lambda^2 \alpha \otimes \beta } \\
		=& \lambda^h \braket{1 - e^{-f} ,\alpha} + (\lambda - \lambda^h) m(\alpha) + \lambda^h \braket{1 - e^{-g} ,\beta} + (\lambda - \lambda^h) m(\beta) \\
		&- \epsilon \lambda^h \braket{e^{\frac{f \oplus g - c}{\epsilon}} - 1, \alpha \otimes \beta} - \epsilon (\lambda^h - \lambda^2) m(\alpha) m(\beta) \\
		=&\lambda^h \OT_\epsilon(\alpha,\beta) + (\lambda - \lambda^h)(m(\alpha) + m(\beta)) - \epsilon (\lambda^h - \lambda^2) m(\alpha)m(\beta).
	\end{align*}
Here as well, the non-homogeneous part cancels when considering the Sinkhorn divergence. 
Note that the linear term involving $(m(\alpha) + m(\beta))$ disappears when adding $- \frac{1}{2} \OT_{\epsilon,\varphi}(\alpha,\alpha) - \frac{1}{2} \OT_{\epsilon,\varphi}(\beta,\beta)$, but adding the mass bias term $\frac{\epsilon}{2}\lambda^2(m(\alpha) - m(\beta))^2$ is required to cancel the product term that involves $m(\alpha) m(\beta)$. 
\end{enumerate}
\end{proof}

\begin{remark}\label{remark:aprox_linear}
\cref{prop:homogeneity_KL_case} can be slightly generalized: whenever the anisotropic proximal operator is linear---$\aprox_{\epsilon,\varphi^*}(p) = \kappa p$ for some $\kappa \in (0,1]$---, the optimal transport plans and the Sinkhorn divergence are $h = \frac{2}{1 + \kappa}$-homogeneous. 
Note that this leads to $\varphi^*(q) = \frac{\epsilon}{\left(\frac{1}{\kappa}-1\right)} \left(e^{\frac{q}{\epsilon} \left(\frac{1}{\kappa}-1\right)}-1\right)$, that is equivalent to use $\rho \KL$ as the marginal penalty with $\rho =  \frac{\epsilon}{\left(\frac{1}{\kappa}-1\right)}$. 
We believe that this condition may be necessary as well: a non-linearity in the $\aprox$ operator prevents $h$-homogeneity to occur and the family of divergences $(\rho \KL)_{\rho \in [0,+\infty]}$ is the only one that makes the UROT problem homogeneous. 
\end{remark}

\subsection{Inhomogeneity in general: the TV case}

The two previous case studies, which fall in the setting ``$\aprox$ is linear'' (see \cref{remark:aprox_linear}), may suggest that the apparent inhomogeneity in the formulation of the unbalanced regularized OT problem does not have much practical aftermaths: the structure of optimal transport plans is preserved and the Sinkhorn divergence is $h$-homogeneous. 
As this encompasses both balanced regularized OT \eqref{eq:balanced_reg_OT_primal} and unbalanced OT using a $\KL$-relaxation of the marginal constraints---arguably covering most applications of regularized OT in practice---this may explain why behaviors related to (in)homogeneity did not receive much attention in the OT community so far. 

In this subsection, we give an example for which inhomogeneity (in particular, of the optimal transport plan) occurs: the case of Total Variation (TV). 
This setting corresponds to taking $\varphi(p) = |1 - p|$, yielding $D_\varphi(\pi_1|\alpha) = \TV(\pi_1 - \alpha)$, $\varphi^*(q) = \max(-1, q)$ and $\aprox_{\epsilon,\varphi^*}(p) = \max(-1, \min(p, 1))$. 
The resulting optimization problem is known as \emph{Optimal Partial Transport} \citep{ot:figalli2010optimal}, a particular case of Unbalanced OT where only a fraction of the total mass of the two measures is transported, and we pay a price proportional to the amount of mass in $\alpha, \beta$ that is not transported. 

The Sinkhorn updates used to produce a sequence $(f^{(\lambda)}_t,g^{(\lambda)}_t)_t$ for the couple of measures $(\lambda \alpha,\lambda \beta)$ read
\begin{equation}\label{eq:sinkhorn_alg_TV_std}
\begin{aligned}
	f_{t+1}^{(\lambda)} &= \min \left( \max\left(-1, - \epsilon \log \braket{e^{\frac{g_t^{(\lambda)} - c}{\epsilon}}, \beta} -  \epsilon \log(\lambda) \right), 1\right),\\
	g_{t+1}^{(\lambda)} &= \min \left( \max \left(-1, -\epsilon \log \braket{e^{\frac{f_{t+1}^{(\lambda)} - c}{\epsilon}}, \alpha} - \epsilon \log(\lambda) \right), 1 \right).
\end{aligned}
\end{equation}
Here, $\aprox_{\epsilon,\varphi^*}$ exhibits sharp changes of behavior when its argument get higher than $1$ (or lower than $-1$). 
This is the source of an inhomogeneous behavior: when the scaling factor $\lambda \to \infty$ (or $\to 0$), this affects the Sinkhorn updates and by consequence the returned (optimal) potentials, transport plan, and Sinkhorn divergence. 

\paragraph{Numerical illustration.} To empirically illustrate the possible inhomogeneous behavior of $\Sk_{\epsilon,\TV}$,  we propose the following experiment. 
We randomly sample two measures $\alpha,\beta$ with $n=5$ and $m=7$ points, respectively and random (non-negative) weights on their support distributed uniformly between $0$ and $1$. 
We then compute the Sinkhorn divergence $\Sk_{\epsilon,\TV}(\lambda \alpha,\lambda\beta)$ for $\lambda \in [1, 100]$ from the optimal dual potentials obtained by iterating \eqref{eq:sinkhorn_alg_TV_std} and the corresponding transport plans through the relation \eqref{eq:relation_primal_dual_std}. 
\cref{fig:TV_std} showcases the dependence of the result on $\lambda$. 
The plot (a) shows that $\Sk_{\epsilon,\TV}$ cannot be $1$-homogeneous. 
If $\Sk_{\epsilon,\TV}$ was $h$-homogeneous for some $h$, one would expect that $\log(\Sk_{\epsilon,\TV}(\lambda\alpha,\lambda\beta)) = h \log(\lambda) + \log(\Sk_{\epsilon,\TV}(\alpha,\beta))$, that would yield a line of slope $h$ in log-log scale. 
Plot (b) in \cref{fig:TV_std} shows that this does not hold overall: a slope break occurs around $\log(\lambda)\sim 2.5$, as a consequence of the non-linearity in $\aprox_{\epsilon,\TV}$. 
This reflects in structural changes in the resulting transport plans as illustrated in the subplots (c,d). 
Computations are run with $\epsilon = 1$. 

\begin{figure}
	\includegraphics[width=\textwidth]{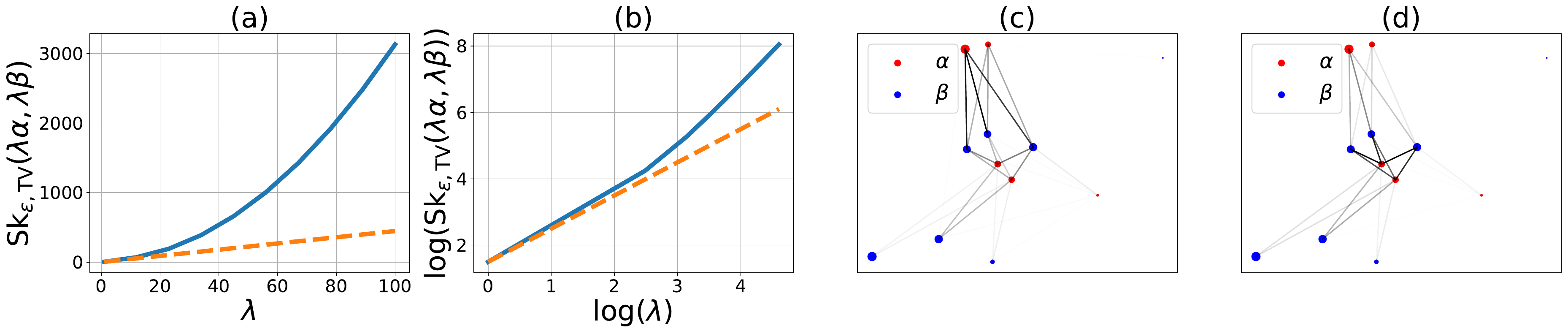}
	\caption{\textbf{Inhomogeneity when using TV as marginal divergence.} \emph{(a)} The Sinkhorn divergence between $\lambda \alpha$ and $\lambda \beta$ for $\lambda \in [1,100]$ for the standard UROT model. Dashed line correspond to the homogeneous behavior $\lambda \cdot \Sk_{\epsilon,\TV}(\alpha,\beta)$. \emph{(b)} The same curve in log-log scale. \emph{(c,d)} The optimal transport plans for $\lambda = 1$ and $\lambda = 100$, respectively. Width of the lines linking $x$ in $\alpha$ to $y$ in $\beta$ are proportional to $\dd\pi(x,y)$. The transport plans are not proportional to each other, showcasing the structural change in the interpolation when rescaling the measures.}
	\label{fig:TV_std}
\end{figure}

\section{An Homogeneous model of Unbalanced Regularized Optimal Transport (HUROT)}
\label{sec:HUROT}

In this section, by slightly changing the entropic regularization term appearing in \eqref{eq:UROT_primal}, we introduce a model of unbalanced regularized OT that presents the advantage of being homogeneous in a very broad setting. 
Let fix $\alpha,\beta \in \MM(\groundspace)$ and assume for now that they have positive total masses: $m(\alpha) > 0,\ m(\beta) > 0$, that is belong to $\MM(\groundspace)\backslash\{0\}$. 
Let also $m_a(\alpha,\beta) \defeq \frac{1}{2}(m(\alpha) + m(\beta))$, $m_g(\alpha,\beta) \defeq \sqrt{m(\alpha)m(\beta)}$ and $m_h(\alpha,\beta) \defeq 2 \left( \frac{1}{m(\alpha)} + \frac{1}{m(\beta)} \right)^{-1}$ denote the arithmetic, geometric and harmonic mean of $m(\alpha)$ and $m(\beta)$, respectively. 
When it is clear from the context, we will simply write $m_a$, $m_g$ and $m_h$ instead.

\begin{definition} For $\pi \in \MM(\groundspace\times \groundspace)$ and $\alpha,\beta \in \MM(\groundspace)\backslash\{0\}$, introduce 
\begin{equation}\label{eq:HOT_regularization_term}
R(\pi|\alpha,\beta) \defeq \frac{1}{2} \left( \KL\left(\pi | \frac{\alpha}{m(\alpha)} \otimes \beta\right) + \KL\left(\pi | \alpha \otimes \frac{\beta}{m(\beta)}\right) \right). 
\end{equation}
The homogeneous unbalanced regularized optimal transport (HUROT) problem between $\alpha$ and $\beta$ is defined as
\begin{equation}\label{eq:HOT_primal}
	\OT_{\epsilon,\varphi}^{[H]}(\alpha,\beta) \defeq \inf_\pi \braket{c,\pi} + D_\varphi(\pi_1 | \alpha) + D_\varphi(\pi_2|\beta) + \epsilon R(\pi | \alpha,\beta).
\end{equation}
\end{definition}

\paragraph{Interpretation.} As detailed in \cref{sec:inhomogeneity_std}, the standard entropic regularization term $\epsilon \KL(\pi | \alpha \otimes \beta)$ introduces an inhomogeneous behavior in the (unbalanced) OT problem. 
When $\alpha, \beta$ are probability measures, using this regularization term is motivated by the fact that the reference measures $\alpha \otimes \beta$ belongs to $\Pi(\alpha,\beta)$, so that the solution of the regularized (balanced) problem \eqref{eq:balanced_reg_OT_primal} interpolates between the exact optimal transport plan ($\epsilon = 0$) and this ``trivial'' one ($\epsilon \to \infty$). 
However, when $\alpha$ and $\beta$ are not probability measures (even if they have the same total masses), $\alpha \otimes \beta \not\in \Pi(\alpha,\beta)$ (its first and second marginals are $m(\beta)\alpha$ and $m(\alpha)\beta$, respectively) and actually, if $m(\alpha) \neq m(\beta)$, there is no measures with $\alpha,\beta$ as marginals. 
The regularization term \eqref{eq:HOT_regularization_term} can be seen as the average of two entropic regularization terms, one whose reference measure has $\beta$ as second marginal, and one whose reference measure has $\alpha$ as first marginal.

\begin{proposition}[Dual formulation]\label{prop:HOT_dual_formulation}
One has:
\begin{equation}\label{eq:HOT_dual}
	\OT_{\epsilon,\varphi}^{[H]}(\alpha,\beta) = \sup_{f,g \in \CC(\groundspace)} \braket{-\varphi^*(-f) , \alpha} + \braket{-\varphi^*(-g), \beta} - \epsilon \braket{ \frac{e^{\frac{f \oplus g - c}{\epsilon}}}{m_g} -\frac{1}{m_h} , \alpha \otimes \beta }.
\end{equation}
Furthermore, if $f,g$ is optimal for \eqref{eq:HOT_dual}, then 
\begin{equation}\label{eq:HOT_primal_dual_relation}
	\pi \defeq \exp\left(\frac{f \oplus g - c}{\epsilon} \right) \frac{\alpha \otimes \beta}{m_g(\alpha,\beta)}
\end{equation}
is an optimal transport plan for the problem \eqref{eq:HOT_primal}.
\end{proposition}

The proof is essentially a variation of the standard proofs of duality in regularized optimal transport which rely on an application of the Fenchel-Rockafellar theorem, see for instance \citep[\S 3.1]{ot:sejourne2019sinkhorn}, \citep[Prop.~4]{ot:genevay2019phd}, \citep[Thm.~1]{ot:chizat2018scaling}. 
The computations that change (due to the modified entropic regularization term) are detailed in the appendix. 

We now state the homogeneity of the HUROT model. 

\begin{proposition}\label{prop:HOT_is_homogeneous}
	$\OT_{\epsilon,\varphi}^{[H]}$ is $1$-homogeneous. 
	Furthermore, if $(f,g)$ is a pair of optimal dual potentials for the couple $(\alpha,\beta)$, then it is also optimal for the couple $(\lambda \alpha,\lambda \beta)$. 
\end{proposition}

\begin{proof}
	The proof simply follows from introducing
	\[ J_{(\alpha,\beta)}^{[H]} (f,g) \defeq \braket{-\varphi^*(-f) , \alpha} + \braket{-\varphi^*(-g), \beta} - \epsilon \braket{ \frac{e^{\frac{f \oplus g - c}{\epsilon}}}{m_g(\alpha,\beta)} -\frac{1}{m_h(\alpha,\beta)} , \alpha \otimes \beta } \]
	and observing that for any $\lambda > 0$, since $m_g(\lambda \alpha,\lambda \beta) = \lambda m_g(\alpha,\beta)$ and $m_h(\lambda \alpha ,\lambda \beta) = \lambda m_h(\alpha,\beta)$, we have
	\[ J^{[H]}_{(\lambda \alpha,\lambda \beta)}(f,g) = \lambda \cdot J_{(\alpha,\beta)}^{[H]}(f,g), \]
	yielding the conclusion.
\end{proof}

\begin{corollary}
	If $\pi$ is an optimal transport plan for the HUROT model \eqref{eq:HOT_primal} for the couple of measures $(\alpha,\beta)$, then $\lambda \pi$ is optimal for the couple $(\lambda \alpha,\lambda\beta)$. 
\end{corollary}

\begin{proof}
This follows from the primal-dual relationship \eqref{eq:HOT_primal_dual_relation} and the fact that $\frac{\lambda \alpha \otimes \lambda \beta}{m_g(\lambda \alpha,\lambda \beta)} = \lambda \frac{\alpha \otimes \beta}{m_g(\alpha,\beta)}$.
\end{proof}

As for the standard model, we can derive first order conditions on the dual that read
\begin{equation}
	\begin{aligned}
		f(x) &= -\aprox_{\epsilon,\varphi^*} \left( \epsilon \log \braket{e^{\frac{g - c(x,\cdot)}{\epsilon}}, \frac{\alpha}{m_g(\alpha,\beta)}} \right), \qquad \alpha-\text{a.e.}\\
		g(y) &= -\aprox_{\epsilon,\varphi^*} \left( \epsilon \log \braket{e^{\frac{f - c(\cdot,y)}{\epsilon}}, \frac{\beta}{m_g(\alpha,\beta)}} \right), \qquad \beta-\text{a.e.},
	\end{aligned}
\end{equation}
yielding the Homogeneous Sinkhorn algorithm:
\begin{equation}\label{eq:HUROT_Sinkhorn_algo}
	\begin{aligned}
		f_{t+1} &= -\aprox_{\epsilon,\varphi^*} \left( \epsilon \log \braket{e^{\frac{g_t - c}{\epsilon}}, \frac{\alpha}{m_g(\alpha,\beta)}} \right),\\
		g_{t+1} &= -\aprox_{\epsilon,\varphi^*} \left( \epsilon \log \braket{e^{\frac{f_{t+1} - c}{\epsilon}}, \frac{\beta}{m_g(\alpha,\beta)}} \right).
	\end{aligned}
\end{equation}
This iterative algorithm can be seen as the standard Sinkhorn algorithm \eqref{eq:sinkhorn_algorithm_std} applied to the renormalized measures $\left( \frac{\alpha}{m_g(\alpha,\beta)},\frac{\beta}{m_g(\alpha,\beta)}, \right)$ and benefits from all the properties proved in \citep{ot:sejourne2019sinkhorn}. 
In particular, it converges toward a fixed point $(f,g)$ that is an optimal couple of potentials for the HUROT model. Numerically, optimal potentials can thus be directly obtained using dedicated software such as \texttt{POT} \citep{flamary2021pot} without requiring further development, and can be re-injected in the objective function $J_{(\alpha,\beta)}^{[H]}$ to get the corresponding homogeneous transport cost $\OT_{\epsilon,\varphi}^{[H]}(\alpha,\beta)$.

\begin{proposition}
	Let $(f_0,g_0) \in \CC(\groundspace)$, $\alpha,\beta$ be two non-zero measures, and $\lambda > 0$. The sequence $(f^{(\lambda)}_t,g^{(\lambda)}_t)_t$ produced by \eqref{eq:HUROT_Sinkhorn_algo} for the couple of measures $(\lambda \alpha,\lambda \beta)$ initialized at $(f_0,g_0)$ is independent of $\lambda$. 
\end{proposition}

\begin{proof}
It is an immediate consequence of \eqref{eq:HUROT_Sinkhorn_algo} and the fact that $\frac{\lambda \alpha}{m_g(\lambda \alpha,\lambda \beta)} = \frac{\alpha}{m_g(\alpha,\beta)}$.
\end{proof}

\begin{proposition}[Continuity of the HUROT model]
	Let $\alpha,\beta \in \MM(\groundspace)\backslash\{0\}$. 
	Consider two sequences $(\alpha_n)_n,(\beta_n)_n$ in $\MM(\groundspace)\backslash\{0\}$ that weakly converge toward $\alpha$ and $\beta$, respectively. 
	
	Then 
	\[ \OT_{\epsilon,\varphi}^{[H]}(\alpha_n,\beta_n) \to \OT_{\epsilon,\varphi}^{[H]}(\alpha,\beta).\]
\end{proposition}

\begin{proof}
We know that $(f_n,g_n)$ is optimal for the HUROT model for the couple $(\alpha_n,\beta_n)$ if and only if it is optimal for the standard model for the couple $\left(\frac{\alpha_n}{m_g(\alpha_n,\beta_n)}, \frac{\beta_n}{m_g(\alpha_n,\beta_n)}\right)$ which converges (as $\alpha_n,\beta_n,\alpha,\beta \neq 0$) to $\left(\frac{\alpha}{m_g(\alpha,\beta)}, \frac{\beta}{m_g(\alpha,\beta)}\right)$. 

Using \citep[Prop.~10 and Thm.~2]{ot:sejourne2019sinkhorn}, it implies in the settings considered in this work ($\varphi = \imath_c$, $\KL$ or $\TV$) that $(f_n,g_n)_n$ converges (uniformly) toward a pair $(f,g)$ that is optimal (in the HUROT model) for the couple $(\alpha,\beta)$ and, by continuity of the objective functional in $(\alpha,\beta,f,g)$ it follows that $\OT_{\epsilon,\varphi}^{[H]}(\alpha_n,\beta_n) \to \OT_{\epsilon,\varphi}^{[H]}(\alpha,\beta)$. 
\end{proof}

We can now introduce the corresponding notion of (homogeneous) Sinkhorn divergence. 

\begin{definition}
	Let $\alpha,\beta \in \MM(\groundspace)$ with $m(\alpha),m(\beta) > 0$. 
	The homogeneous Sinkhorn divergence between $\alpha$ and $\beta$ is defined as
	\begin{equation}
		\Sk_{\epsilon,\varphi}^{[H]}(\alpha,\beta) \defeq \OT_{\epsilon,\varphi^*}^{[H]}(\alpha,\beta) - \frac{1}{2}\OT_{\epsilon,\varphi^*}^{[H]}(\alpha,\alpha) - \frac{1}{2} \OT_{\epsilon,\varphi^*}^{[H]}(\beta,\beta).
	\end{equation}
\end{definition}

By construction, $\Sk_{\epsilon,\varphi^*}^{[H]}$ is homogeneous. 
Interestingly, it is also non-negative under standard assumptions, without needing a ``mass bias'' term (see \cref{remark:mass_bias_std}).

\begin{proposition}\label{prop:HOT_sk_positive}
Let $K_\epsilon(x,y) = e^{-\frac{c(x,y)}{\epsilon}}$, and assume that $K_\epsilon$ is a positive definite kernel. Then, 
\[ \Sk_{\epsilon,\varphi^*}^{[H]}(\alpha,\beta) \geq 0, \]
with equality if, and only if, $\alpha = \beta$.
\end{proposition}

The proof of this proposition rely on the following result, adapted from \citep[Prop.~14]{ot:sejourne2019sinkhorn}. 
For the sake of concision, its proof has been deferred to the appendix.
\begin{lemma}\label{lemma:symmetric_HUROT} One has
	\begin{equation}\label{eq:HOT_self_dual}
	 \OT^{[H]}_{\epsilon,\varphi}(\alpha,\alpha) = \sup_{f \in \CC(\groundspace)} 2 \braket{- \varphi^*(-f), \alpha} - \epsilon \braket{e^{\frac{f \oplus f - c}{\epsilon}} - 1 , \frac{\alpha \otimes \alpha}{m(\alpha)}}
	 \end{equation}
\end{lemma}

\begin{proof}[Proof of \cref{prop:HOT_sk_positive}]
	Let $f_\alpha$ and $g_\beta$ be the minimizers of $\OT^{[H]}_{\epsilon,\varphi}(\alpha,\alpha)$ and $\OT^{[H]}_{\epsilon,\varphi}(\beta,\beta)$, respectively. 
	Note the relation
	\[ \OT^{[H]}_{\epsilon,\varphi}(\alpha,\alpha) = 2 \braket{- \varphi^*(-f_\alpha) , \alpha} - \epsilon \left\|e^{\frac{f_\alpha}{\epsilon}} \frac{\alpha}{\sqrt{m(\alpha)}} \right\|^2_{K_\epsilon} + \epsilon m(\alpha)\]
	and symmetrically in $\beta$. 
	
	As $f_\alpha$ and $g_\beta$ are sub-optimal for the dual problem corresponding to $\OT_{\epsilon,\varphi}^{[H]}(\alpha,\beta)$, we have:
	\begin{align*}
		\OT^{[H]}_{\epsilon,\varphi}(\alpha,\beta) \geq &\braket{-\varphi^*(-f_\alpha),\alpha} + \braket{-\varphi^*(-g_\beta),\beta} - \epsilon \braket{ e^{\frac{f_\alpha \oplus g_\beta - c}{\epsilon}}, \frac{\alpha \otimes \beta}{\sqrt{m(\alpha) m(\beta)}}} + \frac{\epsilon}{2} (m(\alpha) + m(\beta)) \\
		\geq &\braket{-\varphi^*(-f_\alpha),\alpha} + \braket{-\varphi^*(-g_\beta),\beta} - \epsilon \braket{ e^{\frac{f_\alpha}{\epsilon}} \frac{\alpha}{\sqrt{m(\alpha)}}, e^{\frac{g_\beta}{\epsilon}} \frac{\beta}{\sqrt{m(\beta)} }}_{K_\epsilon} + \frac{\epsilon}{2} (m(\alpha) + m(\beta)) \\
		\geq &\frac{1}{2} \OT^{[H]}_{\epsilon,\varphi}(\alpha,\alpha) + \frac{1}{2} \OT^{[H]}_{\epsilon,\varphi}(\beta,\beta) \\
		&+ \frac{\epsilon}{2}  \left\|e^{\frac{f_\alpha}{\epsilon}} \frac{\alpha}{\sqrt{m(\alpha)}} \right\|^2_{K_\epsilon} + \frac{\epsilon}{2} \left\|e^{\frac{g_\beta}{\epsilon}} \frac{\beta}{\sqrt{m(\beta)}} \right\|^2_{K_\epsilon} - \epsilon \braket{ e^{\frac{f_\alpha}{\epsilon}} \frac{\alpha}{\sqrt{m(\alpha)}}, e^{\frac{g_\beta}{\epsilon}} \frac{\beta}{\sqrt{m(\beta)} }}_{K_\epsilon}
	\end{align*}
	so that
	\begin{equation}
	\OT^{[H]}_{\epsilon,\varphi}(\alpha,\beta) - \frac{1}{2} \OT^{[H]}_{\epsilon,\varphi}(\alpha,\alpha) - \frac{1}{2} \OT^{[H]}_{\epsilon,\varphi}(\beta,\beta)	\geq \left\| e^{\frac{f_\alpha}{\epsilon}} \frac{\alpha}{\sqrt{m(\alpha)}} - e^{\frac{g_\beta}{\epsilon}} \frac{\beta}{\sqrt{m(\beta)}}  \right\|_{K_\epsilon} \geq 0
	\end{equation}
	which proves the non-negativity. 
	
	Furthermore, the equality case reads $e^{\frac{f_\alpha}{\epsilon}} \frac{\alpha}{\sqrt{m(\alpha)}} = e^{\frac{g_\beta}{\epsilon}} \frac{\beta}{\sqrt{m(\beta)}}$. 
	By the characterization of $f_\alpha$ and $g_\beta$ as fixed point of their respective Sinkhorn algorithms, we have
	\begin{align*} 
	f_\alpha &= -\aprox_{\epsilon,\varphi^*} \left( \epsilon \log \braket{e^{\frac{f_\alpha - c}{\epsilon}}, \frac{\alpha}{m(\alpha)}} \right),\\
	g_\beta &= -\aprox_{\epsilon,\varphi^*} \left( \epsilon \log \braket{e^{\frac{g_\beta- c}{\epsilon}}, \frac{\beta}{m(\beta)}} \right).
	\end{align*}
	Using the equality case aforementioned, we have 
		\begin{align*} 
	f_\alpha &= -\aprox_{\epsilon,\varphi^*} \left( \epsilon \log \braket{e^{\frac{g_\beta - c}{\epsilon}}, \frac{\beta}{m_g(\alpha,\beta)}} \right),\\
	g_\beta &= -\aprox_{\epsilon,\varphi^*} \left( \epsilon \log \braket{e^{\frac{f_\alpha - c}{\epsilon}}, \frac{\alpha}{m_g(\alpha,\beta)}} \right).
	\end{align*}
	Therefore, $(f_\alpha,g_\beta)$ is actually an optimal couple for the HUROT problem between $\alpha$ and $\beta$, as a fixed point of the corresponding Sinkhorn map.
	
	From this, we can write the optimal transport plans $\pi_{\alpha\beta},\pi_{\alpha\alpha},\pi_{\beta\beta}$ between the corresponding couple of measures as 
	\[ \pi_{\alpha\beta} = e^{\frac{f_\alpha \oplus g_\beta - c}{\epsilon}} \frac{\dd \alpha \otimes \beta}{m_g(\alpha,\beta)},\quad \pi_{\alpha\alpha} = e^{\frac{f_\alpha \oplus f_\alpha - c }{\epsilon}} \frac{\dd \alpha \otimes \alpha}{m(\alpha)},\quad \pi_{\beta \beta} = e^{\frac{g_\beta \oplus g_\beta - c}{\epsilon}} \frac{\dd \beta\otimes \beta}{m(\beta)},\]
	which actually reads
	\[ \pi_{\alpha\beta} = \pi_{\alpha\alpha} = \pi_{\beta\beta}. \]
	Let $\pi$ denote this common transportation plan. 
	Since $\Sk_{\epsilon,\varphi}^{[H]}(\alpha,\beta) = 0$, and observing that the terms $\braket{c,\pi},D_\varphi(\pi_1|\alpha)$ and $D_\varphi(\pi_2|\beta)$ in the primal problems cancel each other, and using the relations
	\begin{align*} 
	&2\KL(\pi | \alpha \otimes \beta) - \KL(\pi | \alpha \otimes \alpha) - \KL(\pi|\beta \otimes \beta) = 0, \\
	&\frac{1}{2}\left( \KL\left(\pi|\frac{\alpha \otimes \beta}{m(\alpha)} \right) + \KL\left(\pi|\frac{\alpha \otimes \beta}{m(\beta)} \right) \right) = \KL(\pi | \alpha \otimes \beta) + m(\pi) \log(m_g(\alpha,\beta)) + m_a(\alpha,\beta) - m(\alpha)m(\beta),
	\end{align*}
	 we can write
	\begin{align*}
		0 = &\frac{1}{2} \left( \KL\left(\pi|\frac{\alpha \otimes \beta}{m(\alpha)}\right) + \KL\left(\pi|\frac{\alpha \otimes \beta}{m(\beta)}\right)\right) - \frac{1}{2} \KL\left(\pi | \frac{\alpha \otimes \alpha}{m(\alpha)}\right) - \frac{1}{2} \KL\left(\pi | \frac{\beta \otimes \beta}{m(\beta)} \right), \\
		= &m(\pi)\log(m_g) + m_a - m(\alpha)m(\beta) \\
		& - \frac{1}{2} m(\pi) \log(m(\alpha)) - \frac{1}{2} m(\alpha) + \frac{1}{2} m(\alpha)^2 \\
		&- \frac{1}{2} m(\pi) \log(m(\beta)) - \frac{1}{2} m(\beta) + \frac{1}{2} m(\beta)^2 \\
		= &\frac{1}{2} (m(\alpha) - m(\beta))^2
	\end{align*}
	which implies that $m(\alpha) = m(\beta) \eqdef m$. 
	From this, it follows that 
	\[
	 f_\alpha = -\aprox_{\epsilon,\varphi^*} \left( \epsilon \log \braket{e^{\frac{f_\alpha - c}{\epsilon}}, \frac{\alpha}{m}} \right) = -\aprox_{\epsilon,\varphi^*} \left( \epsilon \log \braket{e^{\frac{g_\beta- c}{\epsilon}}, \frac{\beta}{m}} \right) = g_\beta, 
	\]
	hence $\alpha = \beta$.
\end{proof}

\begin{remark}
It may be appealing to replace the entropic regularization term \eqref{eq:HOT_regularization_term} by $\epsilon \KL\left( \pi | \frac{\alpha \otimes \beta}{m_g(\alpha,\beta)} \right)$. 
This indeed leads to an homogeneous problem that shares most of the properties of the proposed $\OT_{\epsilon,\varphi}^{[H]}$. 
Actually, the dual formulation would read
\[ \sup_{f,g \in \CC(\groundspace)} \braket{-\varphi^*(-f) , \alpha} + \braket{-\varphi^*(-g), \beta} - \epsilon \braket{ \frac{e^{\frac{f \oplus g - c}{\epsilon}}}{m_g} -\frac{1}{m_g} , \alpha \otimes \beta },  \]
so that the two quantities only differ from a constant term and are substantially equivalent. 
Note also that $\pi \mapsto \frac{\epsilon}{2} \left( \KL\left(\pi | \frac{\alpha}{m(\alpha)} \otimes \beta\right) + \KL\left(\pi | \alpha \otimes \frac{\beta}{m(\beta)}\right) \right)$ is minimized for $\pi = \frac{\alpha \otimes \beta}{m_g(\alpha,\beta)}$, so both entropic terms play morally the same role. 

The one we propose presents the advantage of leading to a Sinkhorn divergence that does not need the introduction of a mass bias term: using $\epsilon \KL\left( \pi | \frac{\alpha \otimes \beta}{m_g(\alpha,\beta)} \right)$ would require to add $+\epsilon(\sqrt{m(\alpha)}-\sqrt{m(\beta)})^2$ to the corresponding Sinkhorn divergence to make it positive. Interestingly, this mass bias correspond to a sort of Hellinger distance between the masses of the two measures.
\end{remark}

\paragraph{Continuity around the null measure.} Previously in this section, we only considered the HUROT model whenever $\alpha,\beta \neq 0$. 
As in the standard case \citep[\S 4.6]{ot:sejourne2019sinkhorn}, assessing continuity of our model around the null measure requires specific care. 
Recall that we assume $\varphi(0) < \infty$. 

\begin{proposition}[Continuity around the null measure]\label{eq:HUROT_continuity_null_measure}~\\
	$\bullet$ Let $\beta \in \MM(\groundspace)\backslash\{0\}$. Define 
	\[ \OT_{\epsilon,\varphi}^{[H]}(0,\beta) \defeq \left(\varphi(0) + \frac{\epsilon}{2}\right) m(\beta). \]
	Let $(\alpha_n)_n$ be a sequence of non-null measures that weakly converges toward the null measure: $\alpha_n \cvweak 0$. Then
	\[ \OT_{\epsilon,\varphi}^{[H]}(\alpha_n,\beta) \to \OT_{\epsilon,\varphi}^{[H]}(0,\beta).\]
	$\bullet$ Furthermore, if we set
		\[ \OT_{\epsilon,\varphi}^{[H]}(0,0) \defeq 0, \]
	then for any sequences $(\alpha_n)_n,(\beta_n)_n$ that both weakly converge toward the null measure, one has
	\[ \OT_{\epsilon,\varphi}^{[H]}(\alpha_n,\beta_n) \to 0.\]
\end{proposition}

\begin{remark}
Contrary to the standard UROT model, $\OT_{\epsilon,\varphi}^{[H]}(0,\beta)$ depends on $\epsilon$ (the result is simply $\varphi(0)m(\beta)$ in the standard model). 
This can be seen as an artifact of the fact that our model directly encompasses the ``mass bias'' in the functional $\OT_{\epsilon,\varphi}^{[H]}$. 
\end{remark}

\begin{proof}[Proof of \cref{eq:HUROT_continuity_null_measure}] The proof where only $\alpha_n \to 0$ follows the spirit of the one of \citep[Prop.~18]{ot:sejourne2019sinkhorn}, though requiring specific adaptation related to our regularization term. 
When both measures go to $0$, we can leverage the homogeneity of our model to prove the claim easily. 

$\bullet$ Using that $\alpha_n \otimes \beta$ is a suboptimal transport plan for \eqref{eq:HOT_primal}, we have
\begin{align*}
	\OT_{\epsilon,\varphi}^{[H]}(\alpha_n,\beta) \leq &\braket{c,\alpha_n \otimes\beta} + D_\varphi(m(\beta) \alpha_n | \alpha_n) + D_\varphi(m(\alpha_n)\beta|\beta) + \epsilon R(\alpha_n \otimes \beta |\alpha_n,\beta) \\
	\leq &\braket{c,\alpha_n \otimes\beta} + m(\alpha_n) \varphi(m(\beta)) + m(\beta) \varphi(m(\alpha_n)) \\
	&+ \epsilon \left(m(\alpha_n)m(\beta) \log(m_g(\alpha_n,\beta)) + \frac{1}{2}(m(\alpha_n) + m(\beta)) - m(\alpha_n)m(\beta)\right) \\
	\to &\varphi(0)m(\beta) + \frac{\epsilon}{2} m(\beta).
\end{align*}
On the other hand, Jensen inequality applied to $D_\varphi$ allows us to write
\begin{align*}
	\OT_{\epsilon,\varphi}^{[H]}(\alpha_n,\beta) \geq & \inf_\pi \braket{c,\pi} + m(\alpha_n)\varphi(m(\pi)) + m(\beta) \varphi(m(\pi)) + \epsilon R(\pi|\alpha_n , \beta) \eqdef F_n(\pi).
\end{align*}
We observe that
\begin{align*}
\lim_{n \to \infty} F_n(\pi)
	\begin{cases} 
		 \geq \braket{c,\pi} + m(\beta)\varphi(m(\pi)) + \frac{\epsilon}{2} \KL(\pi | 0) = +\infty \qquad \text{ if } \pi \neq 0,\\
		= m(\beta) \varphi(0) + \frac{\epsilon}{2}m(\beta) \qquad \text{ if } \pi = 0,
	\end{cases}
\end{align*}
where the second equality follows from the relation 
\[ R(\pi|\alpha_n,\beta) = \KL(\pi|\alpha_n \otimes \beta) - m(\pi)\log(m_g(\alpha_n,\beta)) + m_a(\alpha_n,\beta) - m(\alpha_n)m(\beta) \]
which evaluates to $\frac{1}{2} m(\beta)$ for $\pi = 0$ and $\alpha_n \to 0$. 

As $F_n$ is lower-semicontinuous, it follows that $\lim_n \OT_{\epsilon,\varphi}^{[H]}(\alpha_n,\beta) \geq \varphi(0) m(\beta) + \frac{\epsilon}{2} m(\beta)$, and finally
\[ \lim_{n\to\infty} \OT_{\epsilon,\varphi}^{[H]}(\alpha_n,\beta) = \left(\varphi(0) + \frac{\epsilon}{2}\right) m(\beta), \]
proving the continuity of $\OT_{\epsilon,\varphi}^{[H]}$ around couple of the form $(0,\beta)$ when $\beta \neq 0$.

$\bullet$ We now consider two sequences $\alpha_n, \beta_n \cvweak 0$. Define $M_n = \max(m(\alpha_n),m(\beta_n))$. Using the homogeneity of our model, we can write
\[ \OT_{\epsilon,\varphi}^{[H]}(\alpha_n,\beta_n) = M_n \cdot \OT_{\epsilon,\varphi}^{[H]}\left(\frac{\alpha_n}{M_n}, \frac{\beta_n}{M_n}\right). \]
Using $\frac{\alpha_n}{M_n} \otimes \frac{\beta_n}{M_n}$ as a suboptimal transport plan, we have (note that the two measures have total masses $\leq 1$)
\[ \OT_{\epsilon,\varphi}^{[H]}\left(\frac{\alpha_n}{M_n}, \frac{\beta_n}{M_n}\right) \leq \|c\|_\infty + \varphi\left(\frac{\alpha_n}{M_n}\right) + \varphi\left(\frac{\beta_n}{M_n}\right) + 1. \]
As $\varphi$ is bounded over $[0,1]$, it follows that $\left(\OT_{\epsilon,\varphi}^{[H]}\left(\frac{\alpha_n}{M_n}, \frac{\beta_n}{M_n}\right)\right)_n$ is bounded as well, hence since $M_n \to 0$, 
\[ \lim_{n \to \infty} \OT_{\epsilon,\varphi}^{[H]}(\alpha_n,\beta_n) = 0, \]
proving the continuity in this case as well. 
\end{proof}

\section{Application to Optimal Transport with boundary}
\label{sec:OT_with_boundary}

\subsection{Definition and motivation}

Optimal Transport with Boundary (OTB) was introduced by Figalli and Gigli in \citep{ot:figalli2010newTransportationDistance} as a way to model heat diffusion equations with specific boundary conditions. 
We first give a brief introduction to this model as introduced by the authors in their seminal paper. 

We consider an open bounded domain $\groundspace \subset \R^d$. Let $\overline{\groundspace}$ be its closure and $\thediag$ denote its boundary. 
For the sake of simplicity, we assume that the cost function $c : \overline{\groundspace} \times \overline{\groundspace} \to \R_+$ is given by $c(x,y) = \|x-y\|^2$, though most of the approach developed in the following would adapt to more general symmetric Lipschitz continuous cost functions. 
To alleviate notations, we introduce $\cdiag(x) \defeq c(x,\thediag) = c(\thediag,x) = \inf_{y \in \thediag} c(x,y)$. 
We also assume that the boundary $\thediag$ is regular enough so that there exist a measurable map $P: \groundspace \to \thediag$ such that $c(x,P(x)) = \cdiag(x)$. 

Now, let $\alpha,\beta$ be two locally finite Radon measures supported on $\groundspace$ which can be thought as representing an initial and a final distribution of heat. 
The idea is the following: during the diffusion process, mass (heat) can either move inside the domain ($\groundspace \to \groundspace$) or it may happen that the boundary of the domain $\thediag$ absorbs some mass ($\groundspace \to \thediag$) or redistributes mass to the domain ($\thediag \to \groundspace$). 

Formally, we introduce the set of \emph{admissible plans}
\begin{equation}\label{eq:def_admissible_transport_plan}
 \Adm(\alpha,\beta) \defeq \left\{ \pi \in \MM(\overline{\groundspace} \times \overline{\groundspace}),\ \forall A \subset \groundspace, \pi(A \times \overline{\groundspace}) = \alpha(A),\ \forall B \subset \groundspace, \pi(\overline{\groundspace} \times B) = \beta(B)\right\}.
 \end{equation}
Now, consider the following optimization problem:
\begin{equation}\label{eq:def_FG_problem}
	\FG(\alpha,\beta) = \inf_{\pi \in \Adm(\alpha,\beta)} \iint_{\overline{\groundspace} \times \overline{\groundspace}} c(x,y) \dd \pi(x,y).
\end{equation}
To guarantee that $\FG(\alpha,\beta) <+\infty$, we restrict to measures $\alpha,\beta$ that have finite \emph{total persistence}, where the total persistence of a measure $\mu \in \MM(\groundspace)$ is defined as
\begin{equation}
	\Pers(\mu) \defeq \FG(\mu,0) = \int_\groundspace \cdiag(x) \dd \mu(x).
\end{equation}
We will note by $\MM^c(\groundspace)$ the set of such measures. 

The key idea in the definition of admissible plans \eqref{eq:def_admissible_transport_plan} is that $\pi$ is not constrained on $\thediag \times \thediag$, in contrast with standard (balanced, non-regularized) OT \eqref{eq:vanillaOT}. 
This degree of freedom allows $\thediag$ to play the role of a reservoir that can store and redistribute any amount of mass, provided we pay the corresponding cost $\cdiag(\cdot)$, enabling in particular the comparison of measures with different (and even possibly infinite) total masses.
Note also that when $c(x,y) = \|x-y\|^2$, $(\FG(\cdot,\cdot))^{1/2}$ defines a metric over $\MM^c(\groundspace)$, and the resulting metric space is Polish (complete, separable) provided we allow for measures with infinite total masses. 

\begin{remark}[Links with Topological Data Analysis.] \label{rem:link_OTB_TDA}
This transportation model has not been widely used in OT literature to the best of our knowledge\footnote{In comparison, for instance, to the UROT model presented in \cref{subsec:background_UROT}.}.
However, it has been recently shown in \citep{divol2021understanding} that the metric $\FG$ does exactly coincide with the metrics used by the Topological Data Analysis (TDA) community to compare \emph{Persistence diagrams} (PDs), a type of descriptor routinely used to compare objects with respect to their topological properties, see \citep{tda:edelsbrunner2010computational,tda:chazal2021introduction} for an overview. 
This connection appeared to be fruitful and enabled the adaptation of various tools---both theoretical and computational ones---existing in the OT literature to the context of TDA. 
In a related work \citep{tda:lacombe2018large}, still in the context of TDA, authors proposed a regularized version of \eqref{eq:def_FG_problem} by (substantially) adding a term $+\epsilon \KL(\pi | \LL)$, where $\LL$ denotes the Lebesgue measure on $\groundspace \times \groundspace$. 
However, using the Lebesgue measure (or even $\alpha \otimes \beta$) as reference measure (aside from non-homogeneity) has several drawbacks. 
It is only properly defined for measures with finite total masses, indicating possible problems when the masses of the measures get large in practice---even though the exact distances could be mostly unchanged if the additional mass is close to the boundary $\thediag$. 
In the same vein, it does not follow the spirit of OT with boundary, which tells that points near $\thediag$ have a lesser importance. 
\end{remark}


\subsection{Reformulation as a (spatially varying) unbalanced OT problem} The first step to propose a relevant entropic regularization is to rephrase it in a formalism much closer to the standard UROT model. 
For a given measure $\mu \in \MM^c(\groundspace)$, define the renormalized measure $\hat{\mu}$ by
\[ \forall A \subset \groundspace \text{ Borel},\ \hat{\mu}(A) \defeq \int_A \cdiag(x) \dd \mu(x). \]
Note in particular the relation $m(\hat{\mu}) = \Pers(\mu) < \infty$. 

\begin{proposition}
	Let $\alpha,\beta \in \MM^c(\groundspace)$. Then, 
	\begin{equation}\label{eq:FG_as_UOT}
		\FG(\alpha,\beta) = \inf_{\pi \in \MM(\groundspace \times \groundspace)} \braket{c, \pi} + \int_\upperdiag \varphi\left(x, \frac{\dd \pi_1}{\dd \hat{\alpha}} \right) \dd \hat{\alpha} + \int_\upperdiag \varphi\left(x, \frac{\dd \pi_2}{\dd \hat{\beta}} \right) \dd \hat{\beta},
	\end{equation}
	where 
	\begin{equation}\label{eq:div_spatially_OTB}
			\varphi(x,z) = \begin{cases} |1 - \cdiag(x) \cdot z | &\text{ if } z \in \left[0,\frac{1}{\cdiag(x)}\right] \\ +\infty &\text{ otherwise.} \end{cases}
	\end{equation}
\end{proposition}

\begin{proof}
	Let $\alpha,\beta \in \MM^c(\groundspace)$ and $\pi \in \Adm(\alpha,\beta)$. Without loss of generality, we can assume that $\pi(\thediag \times \thediag)= 0$ \citep[Eq.~(4)]{ot:figalli2010newTransportationDistance} 
	and that $\forall A \subset \groundspace,\ \pi(A \times \thediag) = \pi(A \times P(A))$.
	Let also $\pi_1 = \pi(\cdot \times \groundspace)$ and $\pi_2 = \pi(\groundspace \times \cdot)$, that are the marginals of the restricted plan $\pi_{| \groundspace \times \groundspace}$. 
	Note the constraints $\pi_1 \leq \alpha,\ \pi_2 \leq \beta$. 
	It allows us to write
	\begin{align*}
		\iint_{\overline{\groundspace} \times \overline{\groundspace}} c(x,y) \dd \pi(x,y) &= \iint_{\groundspace \times \groundspace} c(x,y) \dd \pi + \int_{\groundspace \times \thediag} \cdiag(x) \dd \pi + \int_{\thediag \times \groundspace} \cdiag(y) \dd \pi \\
		&= \iint_{\groundspace \times \groundspace} c(x,y) \dd \pi + \int_{\groundspace \times \thediag} \cdiag(x) \dd (\alpha - \pi_1) + \int_{\thediag \times \groundspace} \cdiag(y) \dd (\beta - \pi_2) \\
		&= \iint_{\groundspace \times \groundspace} c(x,y) \dd \pi + \int_{\groundspace \times \thediag} \dd (\hat{\alpha} - \cdiag(x) \pi_1) + \int_{\thediag \times \groundspace} \dd (\hat{\beta} - \cdiag(y) \pi_2) \\
		&= \iint_{\groundspace \times \groundspace} c(x,y) \dd \pi + \int_{\groundspace \times \thediag} \left(1 - \cdiag(x) \frac{\dd \pi_1}{\dd\hat{\alpha}} \right) \dd \hat{\alpha} + \int_{\groundspace \times \thediag} \left(1 - \cdiag(y) \frac{\dd \pi_2}{\dd\hat{\beta}} \right) \dd \hat{\beta} \\
		&= \iint_{\groundspace \times \groundspace} c(x,y) \dd \pi + 	\int_\upperdiag \varphi\left(x, \frac{\dd \pi_1}{\dd \hat{\alpha}} \right) \dd \hat{\alpha} + \int_\upperdiag \varphi\left(x, \frac{\dd \pi_2}{\dd \hat{\beta}} \right) \dd \hat{\beta}.			
	\end{align*}
	From this, we observe that $\pi \in \Adm(\alpha,\beta)$ induces a plan $\pi' = \pi_{|\groundspace \times \groundspace} \in \MM(\groundspace \times \groundspace)$ which implies that 
	\[ \FG(\alpha,\beta) \geq \inf_{\pi' \in \MM(\groundspace \times\groundspace)} \braket{c,\pi'} + \int_\upperdiag \varphi\left(x, \frac{\dd \pi_1}{\dd \hat{\alpha}} \right) \dd \hat{\alpha} + \int_\upperdiag \varphi\left(x, \frac{\dd \pi_2}{\dd \hat{\beta}} \right) \dd \hat{\beta} \eqdef F(\pi').\]
	Conversely, consider $\pi' \in \MM(\groundspace \times \groundspace)$. 
	Let $\pi'_1,\pi'_2$ denote its marginals. 
	Observe that if $\pi'_1 \not\leq \alpha$ or $\pi'_2 \not\leq \beta$, the choice of $\varphi$ implies that $F(\pi') = +\infty$, so we can restrict to such plans. 
	They naturally induce an element $\pi \in \Adm(\alpha,\beta)$ defined by $\pi = \pi'$ on $\groundspace \times \groundspace$, and $\forall A \subset \groundspace, B\subset \thediag,\ \pi(A \times B) = (\alpha - \pi'_1)(P^{-1}(B) \cap A)$ (and symmetrically in $\beta,\pi'_2$), and $\iint_{\overline{\groundspace} \times \overline{\groundspace}} c \dd \pi = F(\pi')$, proving the claim by taking the infimum. 
\end{proof}

This proposition allows us to express $\FG(\alpha, \beta)$ in a formalism much closer to standard (non-regularized) unbalanced OT \eqref{eq:unbalanced_non_reg_OT}: it only involves measures with finite total masses and turns the cost of transporting mass to the boundary $\thediag$ into a penalty between the marginals of $\pi$ and $(\hat{\alpha},\hat{\beta})$. 

\begin{remark}
The key (and essentially sole) difference between \eqref{eq:FG_as_UOT} and \eqref{eq:unbalanced_non_reg_OT} is the dependence of the divergence $\varphi$ on the location $x$, a situation referred to as ``spatially varying divergence'' in \citep[Remark 3]{ot:sejourne2019sinkhorn}. 
This formalism is substantially equivalent to the standard one and most computations adapt seamlessly with the choice of $\varphi$ used in this section. 
The HUROT model could have been presented directly in the more general context of spatially varying divergences in \cref{sec:HUROT}, but this would have required several additional assumptions on $\varphi$ and would have hinder the use of many results of \citep{ot:sejourne2019sinkhorn} directly. 
For the sake of simplicity, we prefer to deal with spatially varying divergences only in this section and for the particular choice \eqref{eq:div_spatially_OTB} of $\varphi$ that allows us to retrieve (when $\epsilon = 0$) the model of Figalli and Gigli \eqref{eq:def_FG_problem}.
\end{remark}

\subsection{Regularized OT with boundary (ROTB)}

\begin{definition}
	Let $\alpha,\beta \in \MM^c(\groundspace)\backslash\{0\}$ and $\epsilon>0$ be a regularization parameter. 
	The corresponding Homogeneous Regularized Optimal Transport with Boundary (ROTB) problem is given by
	\begin{equation}\label{eq:ROTB_primal}
		\FG_\epsilon(\alpha,\beta) \defeq \inf_{\pi \in \MM(\groundspace \times \groundspace)} \braket{c, \pi} + \int_\upperdiag \varphi\left(x, \frac{\dd \pi_1}{\dd \hat{\alpha}} \right) \dd \hat{\alpha} + \int_\upperdiag \varphi\left(x, \frac{\dd \pi_2}{\dd \hat{\beta}} \right) \dd \hat{\beta} + \epsilon R(\pi|\hat{\alpha},\hat{\beta}),
	\end{equation}
	where $\varphi$ is the divergence defined in \eqref{eq:div_spatially_OTB} and $R$ is defined in \eqref{eq:HOT_regularization_term}.
\end{definition}

\begin{remark}
	The formulation \eqref{eq:FG_as_UOT} shows that OT with boundary can be recast as a (spatially varying) UOT problem involving the couple of renormalized measures $(\hat{\alpha},\hat{\beta})$, justifying to use this couple as reference measure in the entropic reference measure in \eqref{eq:ROTB_primal}. 
	Intuitively, it makes the entropic regularization term sensitive to the geometry of the problem, downweighting the points close to the boundary $\thediag$. 
	Formally, the choice of $(\hat{\alpha},\hat{\beta})$ as reference is theoretically supported by the fact that the Sinkhorn divergence corresponding to $\FG_\epsilon$ induces the same convergence as the non-regularized problem \eqref{eq:def_FG_problem}, as detailed below. 
\end{remark}

\paragraph{Dual and Sinkhorn algorithm.} We now give the dual formulation corresponding to \cref{eq:ROTB_primal}. 
A key observation is that despite the primal involves a spatially varying divergence, the dual essentially boils down to a standard problem applied to the renormalized measures $\hat{\alpha}$ and $\hat{\beta}$ in this particular setting, allowing us to adapt the results of \cref{sec:HUROT} seamlessly. 

\begin{proposition}\label{prop:ROTB_dual} Let $\alpha,\beta \in \MM^c(\groundspace)\backslash\{0\}$. One has
	\begin{equation}\label{eq:ROTB_dual}
	\begin{aligned}
	 \FG_\epsilon(\alpha,\beta) = \sup_{f, g \in \CC(\upperdiag)} &\braket{\min(1, f/\cdiag), \hat{\alpha}} + \braket{\min(1, g/\cdiag), \hat{\beta}} \\
	 &- \epsilon \braket{\frac{e^{\frac{f \oplus g - c}{\epsilon}}}{m_g(\hat{\alpha},\hat{\beta})} - \frac{1}{m_h(\hat{\alpha},\hat{\beta})}, \hat{\alpha} \otimes \hat{\beta} }.
	 \end{aligned}
	\end{equation}
	Furthermore, if $(f,g)$ is optimal for \eqref{eq:ROTB_dual}, then
	\[ \pi \defeq \exp\left(\frac{f \oplus g - c}{\epsilon}\right) \frac{\hat{\alpha} \otimes \hat{\beta}}{m_g(\hat{\alpha},\hat{\beta})} \]
	is optimal for the primal problem \eqref{eq:ROTB_primal}.
\end{proposition}

From this dual formulation, we can derive the optimality conditions on dual potentials:
\begin{equation}
\begin{split}
	f(x) = \min\left(\cdiag(x),  - \epsilon \log \left( \cdiag(x) \braket{ e^{\frac{g - c(x, \cdot)}{\epsilon}}, \frac{\hat{\beta}}{m_g(\hat{\alpha},\hat{\beta})}} \right) \right), \ \hat{\alpha}\text{-ae} \\
	g(y) = \min\left( \cdiag(y), - \epsilon \log \left( \cdiag(y) \braket{ e^{\frac{f - c(\cdot, y)}{\epsilon}}, \frac{\hat{\alpha}}{{m_g(\hat{\alpha},\hat{\beta})}}} \right) \right), \ \hat{\beta}\text{-ae},
\end{split}
\end{equation}
and thus define the corresponding Sinkhorn algorithm
\begin{equation}\label{eq:ROTB_sinkhorn}
\begin{split}
	f_{t+1}(x) \defeq \min\left(\cdiag(x),  - \epsilon \log \left( \cdiag(x) \braket{ e^{\frac{g - c(x, \cdot)}{\epsilon}}, \frac{\hat{\beta}}{m_g(\hat{\alpha},\hat{\beta})}} \right)\right) \\
	g_{t+1}(y) \defeq \min\left(\cdiag(y), - \epsilon \log \left( \cdiag(y) \braket{ e^{\frac{f - c(\cdot, y)}{\epsilon}}, \frac{\hat{\beta}}{m_g(\hat{\alpha},\hat{\beta})}} \right)\right),
\end{split}
\end{equation}

Finally, we introduce the corresponding notion of Sinkhorn divergence:
\begin{equation}\label{eq:ROTB_skdiv}
	\Sk\FG_\epsilon(\alpha,\beta) \defeq \FG_\epsilon(\alpha,\beta) - \frac{1}{2} \FG_\epsilon(\alpha,\alpha) - \frac{1}{2} \FG_\epsilon(\beta,\beta).
\end{equation}

This problem enjoys the same properties as the HUROT model introduced in \cref{sec:HUROT}, as summarized in the following proposition.

\begin{proposition}[Properties of ROTB]~\label{prop:HUROT_properties}
	\begin{enumerate}
		\item $\FG_\epsilon$ and $\Sk\FG_\epsilon$ are $1$-homogeneous. The sequence of potentials produced by \eqref{eq:ROTB_sinkhorn} for the couple of measures $(\lambda \alpha,\lambda \beta)$ is independent of $\lambda$, so are the optimal potentials. If $\pi$ is an optimal plan for $(\alpha,\beta)$, $\lambda \pi$ is optimal for the couple $(\lambda \alpha,\lambda \beta)$. 
		\item $\FG_\epsilon$ is continuous with respect to the weak convergence of the renormalized measures: $\widehat{\alpha_n} \cvweak \hat{\alpha} \Rightarrow \FG_\epsilon(\alpha_n,\beta) \to \FG_\epsilon(\alpha,\beta)$. This holds in particular around the null measure by setting $\FG_\epsilon(0,\beta) \defeq \left(1 + \frac{\epsilon}{2}\right) \Pers(\beta)$ and $\FG_\epsilon(0,0) = 0$. 
		\item Under the same assumptions as in \cref{prop:HOT_sk_positive}, $\Sk\FG_\epsilon(\alpha,\beta) \geq 0$, with equality if and only if $\alpha = \beta$.  
	\end{enumerate}
\end{proposition}

The proof directly adapts from the corresponding ones in \cref{sec:HUROT}. 

\begin{proposition}\label{prop:metrize_cv_weak}
	$\FG_\epsilon$ induces the same notion of convergence as $\FG$, that is, for any sequence $(\alpha_n)_n \in \MM^c(\groundspace)^\N$ and any $\alpha \in \MM^c(\groundspace)$, 
	\[ \Sk\FG_\epsilon(\alpha_n,\alpha) \to 0 \Leftrightarrow \widehat{\alpha_n} \cvweak \hat{\alpha} \Leftrightarrow \Sk\FG(\alpha_n, \alpha) \to 0. \]
\end{proposition}

\begin{proof}[Proof of \cref{prop:metrize_cv_weak}.]
The fact that $\widehat{\alpha_n} \cvweak \hat{\alpha} \Leftrightarrow \Sk\FG(\alpha_n, \alpha) \to 0$ is already known \citep[Cor.~3.2]{divol2021understanding}. 
Therefore, it remains to show that $\Sk\FG_\epsilon(\alpha_n,\alpha) \to 0 \Leftrightarrow \widehat{\alpha_n} \cvweak \hat{\alpha}$. 

The converse implication is given by the continuity of $\Sk\FG_\epsilon$ with respect to the weak convergence of the normalized measures (\cref{prop:HUROT_properties}). 
Now, assume that $\Sk\FG_\epsilon(\alpha_n,\alpha) \to 0$. 
If the sequence $(\widehat{\alpha_n})_n$ has uniformly bounded mass (i.e.~$(\alpha_n)_n$ has uniformly bounded total persistence), we know that it must be compact with respect to the weak convergence (as $\groundspace$ is bounded). If so, extracting a converging subsequence converging to some limit $\widehat{\alpha_\infty}$ yields by continuity $\Sk\FG_\epsilon(\alpha_\infty, \alpha) = 0$ and thus $\alpha_\infty = \alpha$. 
This makes $(\widehat{\alpha_n})_n$ a compact sequence with $\hat{\alpha}$ as unique limit, implying $\widehat{\alpha_n} \cvweak \hat{\alpha}$. 

Therefore, it remains to show that $\Sk\FG_\epsilon(\alpha_n,\alpha) \to 0 \Rightarrow \sup_n \Pers(\alpha_n) = \sup_n m(\widehat{\alpha_n}) < +\infty$. 
Let $f_n$ denotes the optimal symmetric potential for the dual problem \eqref{eq:ROTB_dual} corresponding to the couple $(\alpha_n,\alpha_n)$, and $f_{\hat{\alpha}}$ be the one corresponding to the couple $(\alpha,\alpha)$. 
The optimality condition on $f_n$ gives $f_n \geq - \cdiag \geq - L$, where $L = \diam(\groundspace)$. 
Assume first that $\alpha_n, \alpha \neq 0$. 
One has
\[ \Sk\FG_\epsilon(\alpha_n,\alpha) \geq \frac{\epsilon}{2} \| e^{\frac{f_n}{\epsilon}} \frac{\widehat{\alpha_n}}{\sqrt{m(\widehat{\alpha_n})}} - e^{\frac{f_{\hat{\alpha}}}{\epsilon}} \frac{\widehat{\alpha}}{\sqrt{m(\widehat{\alpha})}}\|_{K_\epsilon}. \]
Since $\Sk\FG_\epsilon(\alpha_n,\alpha) \to 0$, one has $\sup_n \left\| e^{\frac{f_n}{\epsilon}} \frac{\widehat{\alpha_n}}{\sqrt{m(\widehat{\alpha_n})}} \right\|_{K_\epsilon} < \infty$, 
and since $(f_n)_n$ is (uniformly) lower bounded, necessarily, $(m(\widehat{\alpha_n}))_n$ is bounded, proving the claim. 
If $\alpha = 0$, the same reasoning yields $e^{\frac{f_n}{\epsilon}} \frac{\widehat{\alpha_n}}{\sqrt{m(\widehat{\alpha_n})}} \cvweak 0$, thus $\sup_n m(\widehat{\alpha_n}) < \infty$ and $\alpha_n \cvweak 0$. 
\end{proof}

\paragraph{Numerical illustration.} To showcase the importance of using an homogeneous model in the context of OT with boundary, we propose the following experiment. 
Inspired by the context of Topological Data Analysis (see \cref{rem:link_OTB_TDA}), we consider the half-plane $\groundspace = \{(t_1,t_2),\ t_1 < t_2\} \subset \R^2$ hence $\thediag = \{(t,t),\ t \in \R\}$.\footnote{Note that $\groundspace$ is not bounded here, but in practice the measures of interest (the so-called persistence diagrams or persistence measures) are usually supported on a bounded subset of $\groundspace$.}
We then sample two measures $\alpha,\beta$ with $n = 5$ and $m = 10$ points respectively, and with weight $1$ on each point. 
We then compute the OTB Sinkhorn divergence $\Sk\FG_\epsilon(\lambda \alpha,\lambda\beta)$ for $\lambda \in [0.01, 100]$ using our homogeneous model and the Sinkhorn divergence one would obtain using the standard UROT model (via the iterations \eqref{eq:sinkhorn_algorithm_std}). 
\cref{fig:OTB_distances} showcases the dependence of the result on $\lambda$. 
As expected, our model exhibits $1$-homogeneity. 
In contrast, the standard model yields a highly inhomogeneous behavior which reflects in many structural changes in the resulting transport plans as showcased in \cref{fig:OTB_transport}. 
Computations are run with $\epsilon = 1$.

\begin{figure}[H]
	\center
	\includegraphics[width=0.8\textwidth]{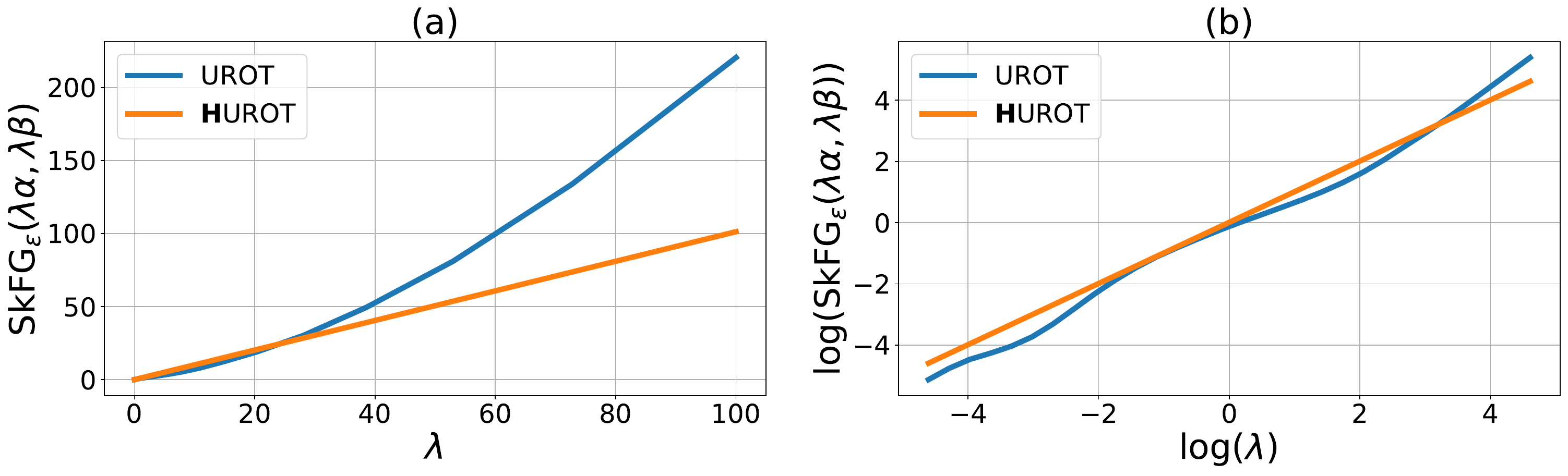}
	\caption{\textbf{Importance of homogeneity for the OTB model.} \emph{(a)} The evolution of $\FG_\epsilon(\lambda \alpha,\lambda \beta)$ for $\lambda \in [0.01, 100]$ using either our homogeneous regularization term \eqref{eq:HOT_regularization_term} (\textbf{H}UROT) or the standard one $+\epsilon \KL(\pi|\hat{\alpha}\otimes\hat{\beta})$ (UROT). As expected, the HUROT model yields a straight line of slope $1$. \emph{(b)} Same curve in log scale. If the standard model was $h$-homogeneous for some $h$ (it is clearly not $1$-homogeneous from the plot (a)), one would expect to observe a line of slope $h$ here. The various slope breaks, due to the non-linearity in the Sinkhorn iterations in this setting, illustrate a highly non-homogeneous behavior.}
	\label{fig:OTB_distances}
\end{figure}

\begin{figure}[H]
	\center
	\includegraphics[width=\textwidth]{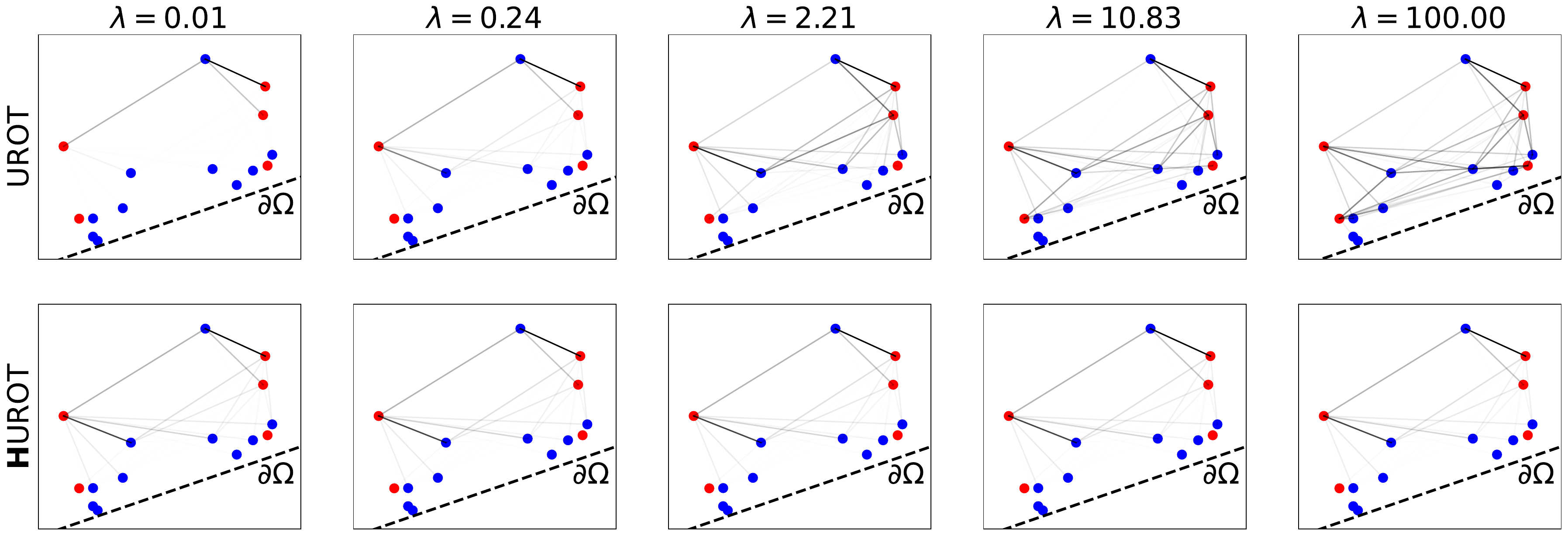}
	\caption{\textbf{Impact of inhomogeneity on the transport plan for the OTB model.} \emph{(Top row)} The transport plans obtained for the couple $(\lambda \alpha,\lambda \beta)$ for varying $\lambda$ using the standard UROT model. Inhomogeneity reflects in structural changes in the resulting transport plan; roughly, increasing $\lambda$ tends to overweight transportation near the boundary $\thediag$. \emph{(Bottom row)} The transportation plans using the HUROT model. As expected, varying $\lambda$ only rescale the transport plan.}
	\label{fig:OTB_transport}
\end{figure}


\section{Discussion}

We believe that the homogeneous UROT model we propose can provide a good alternative to the standard model of Unbalanced Regularized Optimal Transport proposed by Séjourné et al., especially when \textit{(i)} the marginal divergence induces a ``cut-off'' as do the Total Variation or spatially varying divergences involved in OT with boundary and when \textit{(ii)} the masses of the measures considered may be ill-defined (e.g.~depend on the choice of a unit of measurement) or may largely vary on the considered sample. 

It is worth noting that enforcing homogeneity in the regularization term comes with some price as well. 
In particular, in contrast with the standard UROT model, at fixed $\beta$, the map $\alpha \mapsto \OT_{\epsilon,\varphi}^{[H]}( \alpha, \beta)$ is \emph{a priori} not convex with respect to linear interpolation of measures $(1-t)\alpha + t\alpha'$. 
Since the resulting homogeneous Sinkhorn divergence still shares key properties with the standard one, wondering whether there exist a convex reparametrization of $\OT_{\epsilon,\varphi}^{[H]}$ is an important question. 
Other type of convexity properties, for instance along the interpolation curves described by the optimal transport plans, may also be investigated. 

Finally, the development of different numerical tools in the context of OT with boundary, in particular regularized Fréchet means, is a natural follow-up of this work. 
Note that in the context of topological data analysis (which is related to OT with boundary, see \cref{rem:link_OTB_TDA}), regularized barycenters for persistence diagrams have been developed \citep{tda:lacombe2018large}. 
However, the proposed approach uses the Lebesgue measure as reference measure in their entropic regularization term. 
This yields points near the boundary of the space, which tend to outnumber farther points in applications, to outweigh them as well. 
Using instead the reweighted measures $\hat{\alpha},\hat{\beta}$ and our homogeneous formulation is likely to improve the quality of the numerical results that can be obtained.

\paragraph{Acknowledgments} The author wants to thank V.~Divol, T.~Séjourné and F.-X.~Vialard for fruitful discussions that contributed to the development this work. 

\bibliographystyle{abbrv}
\bibliography{biblio}

\clearpage
\appendix

\section{Delayed proofs}

\begin{proof}[Proof of \cref{prop:HOT_dual_formulation}]
The only computations that differ from the proof of duality appearing in \citep{ot:sejourne2019sinkhorn} are those corresponding to our slightly modified entropic regularization term.

Introduce $\xi \defeq \frac{\dd \pi}{\dd \alpha \otimes \beta}$ to alleviate notations. 
\begin{align*}
	\frac{\epsilon}{2} \left( \KL(\pi | \frac{\alpha}{m(\alpha)} \otimes \beta) + \KL(\pi | \alpha \otimes \frac{\beta}{m(\beta)})\right) = &\frac{\epsilon}{2} (\braket{ \xi \log(\xi) - \xi + \log(m(\alpha)) \xi + \frac{1}{m(\alpha)} , \alpha \otimes \beta} \\
	& + \braket{ \xi \log(\xi) - \xi + \log(m(\beta)) \xi + \frac{1}{m(\beta)} , \alpha \otimes \beta} ) \\
	= & \epsilon \braket{\xi \log(\xi) - \xi + \log(\sqrt{m(\alpha)m(\beta)}) \xi + \frac{1}{2} \left(\frac{1}{m(\alpha)} + \frac{1}{m(\beta)} \right) , \alpha \otimes \beta } \\
	= & \epsilon \braket{\xi \log(\xi) - \xi + \log(m_g) \xi + \frac{1}{m_h} , \alpha \otimes \beta }.
\end{align*}

In order to obtain the primal-dual relationship, we write
\begin{align*}
	&- \sup_{\pi} \braket{f \oplus g, \pi} - \braket{c, \pi} - \frac{\epsilon}{2} ( \KL(\pi | \frac{\alpha}{m(\alpha)} \otimes \beta) + \KL(\pi | \alpha \otimes \frac{\beta}{m(\beta)}) )\\
	= & \inf_{\xi} \braket{- (f \oplus g - c) \xi, \alpha \otimes \beta} + \epsilon \braket{\xi \log(\xi) - \xi + \log(m_g) \xi + \frac{1}{m_h} , \alpha \otimes \beta } \\
	= & \inf_\xi \braket{- (f \oplus g - c) \xi + \epsilon ( \xi \log(\xi) - \xi + \log(m_g) \xi + \frac{1}{m_h} , \alpha \otimes \beta }.
\end{align*}
This optimization problem in $\xi$ yields the primal-dual relation \eqref{eq:HOT_primal_dual_relation}.
\begin{equation}\label{eq:relation_primal_dual}
 \xi = \frac{1}{m_g} e^{\frac{f \oplus g - c}{\epsilon}},
 \end{equation}
so that the term $\frac{\epsilon}{2} \left( \KL\left(\pi | \frac{\alpha}{m(\alpha)} \otimes \beta\right) + \KL\left(\pi | \alpha \otimes \frac{\beta}{m(\beta)}\right) \right)$ is equal to
\begin{align*}
	&- (f \oplus g - c) \frac{e^{\frac{f \oplus g - c}{\epsilon}}}{m_g} + \frac{e^{\frac{f \oplus g - c}{\epsilon}}}{m_g} (f \oplus g - c) - \epsilon \log(m_g) \frac{e^{\frac{f \oplus g - c}{\epsilon}}}{m_g} - \epsilon \frac{e^{\frac{f \oplus g - c}{\epsilon}}}{m_g} + \epsilon \log(m_g) \frac{e^{\frac{f \oplus g - c}{\epsilon}}}{m_g} + \epsilon \frac{1}{m_h} \\
	= & - \epsilon \left( \frac{e^{\frac{f \oplus g - c}{\epsilon}}}{m_g} -\frac{1}{m_h} \right).
\end{align*}
Eventually
\begin{equation}
	\OT_\epsilon(\alpha,\beta) = \sup_{f,g} \braket{-\varphi^*(-f) , \alpha} + \braket{-\varphi^*(-g), \beta} - \epsilon \braket{ \frac{e^{\frac{f \oplus g - c}{\epsilon}}}{m_g} -\frac{1}{m_h} , \alpha \otimes \beta }.
\end{equation}
\end{proof}

\begin{proof}[Proof of \cref{lemma:symmetric_HUROT}]
	Let $f \in \CC(\groundspace)$ be optimal in \eqref{eq:HOT_self_dual}. 
	Using the couple $(f,f)$ in \eqref{eq:HOT_dual}, we get
	\[ 	 \OT^{[H]}_{\epsilon,\varphi}(\alpha,\alpha) \geq \sup_{f \in \CC(\groundspace)} 2 \braket{- \varphi^*(-f), \alpha} - \epsilon \braket{e^{\frac{f \oplus f - c}{\epsilon}} - 1 , \frac{\alpha \otimes \alpha}{m(\alpha)}}. \]
	Now, let $\pi = \exp\left(\frac{f \oplus f - c}{\epsilon}\right) \frac{\dd \alpha \otimes \alpha}{m(\alpha)}$. 
	By the symmetry of $c$, its marginals are given by $\pi_1 = \pi_2 = \braket{e^{\frac{f-c}{\epsilon}},\frac{\alpha}{m(\alpha)}} e^{f/\epsilon} \alpha$. 
	As $\pi$ is suboptimal in \eqref{eq:HOT_primal}, we get
	\[ \OT_{\epsilon,\varphi}^{[H]}(\alpha,\alpha) \leq \braket{\pi,c} + 2 D_\varphi(\pi_1|\alpha) + \epsilon \KL\left( \pi | \frac{\alpha \otimes \alpha}{m(\alpha)} \right). \]
	Now, observe that for $i \in \{1,2\}$,
	\[ \frac{\dd \pi_i}{\dd \alpha} = \braket{e^{\frac{f-c}{\epsilon}},\frac{\alpha}{m(\alpha)}} e^{f/\epsilon} \in \partial \varphi^*(-f), \]
	and since $\varphi^*(q) = \sup_p pq - \varphi(p)$, we have that $\forall x \in \groundspace,\ \varphi^*(-f(x)) = -f(x) \frac{\dd \pi_1}{\dd \alpha} - \varphi\left(\frac{\dd \pi_1}{\dd \alpha}\right)$. Therefore, 
	\[ D_\varphi(\pi_1 | \alpha) = \braket{\varphi\left(\frac{\dd \pi_1}{\dd\alpha}\right), \alpha} = \braket{-f \frac{\dd \pi_1}{\dd\alpha} - \varphi^*(-f), \alpha } = -\braket{f,\pi_1} + \braket{-\varphi^*(-f),\pi_1}. \]
	On the other hand, denoting $\zeta = \exp\left(\frac{f \oplus f - c}{\epsilon}\right)$, we have
	\begin{align*}
		\epsilon \KL\left(\pi| \frac{\alpha \otimes \alpha}{m(\alpha)} \right) &= \epsilon \braket{ \log(\zeta) \zeta - \zeta + 1 ,\frac{\alpha \otimes \alpha}{m(\alpha)}} \\
		&= \braket{ f \oplus f - c, \pi} - \epsilon \braket{e^{\frac{f \oplus f - c}{\epsilon}} - 1, \frac{\alpha \otimes \alpha}{m(\alpha)}} \\
		&= 2 \braket{f, \pi_1} - \braket{c,\pi} - \epsilon \braket{e^{\frac{f \oplus f - c}{\epsilon}} - 1, \frac{\alpha \otimes \alpha}{m(\alpha)}}.
	\end{align*}
	Summing the terms together yields the result. 
\end{proof}

\end{document}